\documentclass[reqno, english]{amsart}
\usepackage{amsfonts, amsmath, latexsym, verbatim, amscd, mathrsfs, color, array, titletoc}
\usepackage[colorlinks,
linkcolor=red,
anchorcolor=blue,
citecolor=blue]{hyperref}

\allowdisplaybreaks

\usepackage{bbm}
\usepackage{float}
\usepackage{amsmath, amssymb, amsthm, amsfonts, graphicx, color, mathtools}
\usepackage[hmargin=2.5cm, vmargin=2.5cm]{geometry}
\usepackage{bbm}
\usepackage{pdfsync}
\usepackage{epstopdf}
\usepackage{cite}
\usepackage{multirow}
\usepackage[font=bf,aboveskip=15pt]{caption}
\usepackage[toc,page]{appendix}

\usepackage{enumerate}
\usepackage{bookmark}
\DeclareFontShape{OT1}{cmr}{bx}{sc}{<-> cmbcsc10}{}
\usepackage{ulem}
\usepackage{syntonly}
\usepackage{upgreek}
\usepackage{bm}
\usepackage{nicematrix}
\usepackage{lipsum}

\allowdisplaybreaks[4]

\newcommand{\h}{\hat  }

\newcommand{\cuad}{{\sqcap\kern-.68em\sqcup}}

\newcommand{\be}{\begin{equation}}
	\newcommand{\ee}{\end{equation}}

\newtheorem{lemma}{Lemma}[section]
\newtheorem{prop}{Proposition}[section]
\newtheorem{theorem}{Theorem}[section]

\newtheorem{remark}{Remark}[section]
\newcommand{\bremark}{\begin{remark} \em}
	\newcommand{\eremark}{\end{remark} }

\numberwithin{equation}{section}
\begin{document}
	\title[Lotka-Volterra competition model]{Bifurcation for the Lotka-Volterra competition model}
	
	\author[Z. Li]{Zaizheng Li}
	\address{\noindent
		School of Mathematical Sciences, Hebei Normal University, Shijiazhuang 050024, Hebei, P.R. China}
	\email{zaizhengli@hebtu.edu.cn}

	\author[S. Terracini]{Susanna Terracini}
	\address{\noindent
	Dipartimento di Matematica “Giuseppe Peano”, Universit\`a di Torino, Via Carlo Alberto, 10, 10123 Torino, Italy}
	\email{susanna.terracini@unito.it}

	
	\dedicatory{Dedicated to Nina Uraltseva with great admiration }
	\begin{abstract}
		We analyze the bifurcation phenomenon for the  following two-component competition system:
 	\begin{equation*}
 		\begin{cases}
 			-\Delta u_1=\mu u_1(1-u_1)-\beta \alpha u_1u_2,& \text{in}\ B_1\subset \mathbb{R}^N,\\
 			-\Delta u_2=\sigma u_2(1-u_2)-\beta \gamma u_1u_2,& \text{in}\ B_1\subset \mathbb{R}^N,\\
 			\frac{\partial u_1}{\partial n}=	\frac{\partial u_2}{\partial n}	=0,&\text{on}\ \partial B_1,
 		\end{cases}
 	\end{equation*}
 	where $N\ge 2$, $\alpha>\gamma>0$, $\sigma\ge\mu>0$ and $\beta>\frac{\sigma}{\gamma}$.  More precisely, treating $\beta$ as the bifurcation parameter, we initially perform a local bifurcation analysis around the positive constant solutions, obtaining precise information of where bifurcation could occur, and determine the direction of bifurcation. As a byproduct, the instability of the constant solution is provided. Furthermore, we extend our exploration to the global bifurcation analysis. 
	 Lastly, under the condition $\sigma=\mu$, we demonstrate the limiting configuration on each bifurcation branch as the competition rate $\beta\rightarrow+\infty$.
		\\
		\medskip
		
	\end{abstract}

	\maketitle
	
	%\tableofcontents

\section{Introduction and main results}
 The paper intends to enhance our comprehension of spatial patterns in diffusive Gause-Lotka-Volterra reaction-diffusion systems: 
 	\begin{equation*}%\label{systemLK}
 	\partial_tu_i- d_i\Delta u_i=f_i(u_i)-\sum_{j\neq i}a_{ij} u_iu_j\;,
 \end{equation*}	
where $u_i(x,t)$ denotes the the non-negative population density of the $i$-th species at time $t$ within the spatial domain $\Omega$, $d_i\nabla u_i$ represents the dispersal-induced flow, $f_i(u_i)$ captures the internal growth dynamics of the $i$-th species, and $A=(a_{ij})_{ij}$ describes interspecific interactions, with $a_{ij}$ signifying the interaction rate between species $i$ and $j$. The negative sign denotes competitive interactions. In ecology, these models are instrumental in elucidating the spatial distribution of species and the emergence of ecological patterns (see, e.g., \cite{CanCosBook}). For systems involving three or more species, a wealth of literature exists, such as \cite{CTV-2005,  Dancer-Du-1995, Dancer-Du-1995-2, Dancer-2012, susanna-2024} and the references therein.

 We direct our attention to the scenario of a competing two-species model within a homogeneous environment.  This setting entails the presence of constant coefficients $d_i, a_{ij}$,  and no flux at the boundary (Neumann boundary conditions).   More precisely, 
 \begin{equation}\label{2-sys}
 	\begin{cases}
 	\partial_t u_1-d_1\Delta u_1=f_1(u_1)-a_{12}u_1u_2,
 	\\
 		\partial_t u_2-d_2\Delta u_2=f_2(u_2)-a_{21}u_1u_2,
 	\end{cases}
 \end{equation}
 where $a_{12}, a_{21}>0$.
 An important inquiry arises: \textit{Can competing species coexist in spatially homogeneous environments by adopting non-constant distributions, or is eventual extinction inevitable for one or more of them as $t\rightarrow\infty$? }Specifically, given that the system allows for constant solutions over space and time, the question is whether these solutions are stable and if there exist alternative solutions that avoid the eventual extinction of some populations. This inquiry resonates with Gause's law and Grinnell's competitive exclusion principle in ecology.
 
 Several studies have investigated the competitive exclusion principle and the dynamics of the time-dependent system \eqref{2-sys}, with relevant literature summarized as follows.
 Brown \cite{Brown-1980} established the global asymptotic stability of the constant solution under small competition rates and Neumann boundary conditions.  Pao \cite{Pao-1981} explored coexistence and stability issues under general boundary conditions using comparison theorems. Kishimoto \cite{Kishimoto-1981} demonstrated the instability of any non-constant equilibrium solution for the rectangular parallelepiped domain under Neumann boundary conditions and for the entire space under periodic boundary conditions. After that,  Kishimoto-Weinberger \cite{Kishimoto-1985} extended this result to convex domains. Matano-Mimura \cite{Matano-Mimura-1983} illustrated the existence of a stable spatially-inhomogeneous equilibrium solution for a non-convex domain under Neumann boundary conditions.
Kahane \cite{Kahane-1992} investigated the asymptotic behaviour of solutions as $t\rightarrow\infty$ when $a_{12}=a_{21}$ under Neumann boundary conditions. Dancer-Hilhorst-Mimura-Peletier \cite{Dancer-1999} derived the spatial segregation limit as the competition rate tends to infinity and proved that the limiting system becomes a free boundary problem under Neumann boundary conditions. Later, Crooks et al. \cite{Dancer-2004} obtained similar results under inhomogeneous Dirichlet boundary conditions.

 Regarding the stationary problem, several studies have contributed valuable insights.  Leung \cite{Leung-1980} addressed Dirichlet boundary data and established asymptotic stability. Dancer \cite{Dancer-1984, Dancer-1985, Dancer-1991} employed degree theory in cones to investigate sufficient and necessary conditions for the existence of positive solutions under Dirichlet boundary conditions.  Cosner-Lazer \cite{Cosner-Lazer-1984} derived sufficient conditions for coexistence, uniqueness and stability under Dirichlet boundary conditions. Gui-Lou \cite{Gui-Lou-1994} studied the uniqueness and multiplicity of coexistence states under Dirichlet and Neumann boundary conditions. Gui \cite{Gui-1995} investigated the multiplicity and stability of coexistence states in the case $a_{12}=a_{21}$, under Dirichlet boundary conditions. Conti-Terracini-Verzini \cite{CTV-2005} described spatial segregation limits and provided precise estimates for the rate convergence as the competition rate tends to infinity under inhomogeneous Dirichlet boundary conditions.
 
In the context of bifurcation theory for stationary systems, most literature focuses on the scenario where $f_i(s)=a_i s-b_i s^2$, $i=1, 2$. Blat-Brown \cite{Blat-Brown-1984}  treated $a_2$ as a bifurcation parameter while keeping $a_1$ fixed, exploring the existence of solutions under both Dirichlet and Neumann boundary conditions. Subsequently, Cantrell-Cosner \cite{Cantrell-Cosner-1987} obtained a more detailed analysis by considering the system as a two-parameter ($a_1, a_2$)  bifurcation problem under Dirichlet boundary conditions. For the system with $a_{1}=a_{2}$, $a_{12}\neq a_{21}$,  Gui-Lou \cite{Gui-Lou-1994} studied the global bifurcation phenomenon under Dirichlet and Neumann boundary conditions.  In cases where $a_{1}=a_{2}, a_{12}=a_{21}=b$, Gui \cite{Gui-1995} explored the multiple coexistence states and provided a clear nodal set analysis for global bifurcation branches under Dirichlet boundary conditions, utilizing $b$ as the bifurcation parameter. For more general bifurcation theory, we refer to \cite{CR-1971, CR-1973, CR-1977, R-1971}. 

In our current study, we focus specifically on the case $a_{1}\le a_{2}$, $a_{12}=\beta\alpha>\beta\gamma=a_{21}$, with $\beta$ serving as the bifurcation parameter.  Initially, we conduct a local bifurcation analysis around the positive constant solutions. This involves obtaining accurate estimates of potential bifurcation locations and determining the direction in which bifurcation occurs. Additionally, we provide the instability of the constant solution.
Furthermore, we delve into the exploration of global bifurcation branches. %Through rigorous analysis, we establish that each branch is unbounded. 
Finally, in the case $\sigma=\mu$, we describe the asymptotic behaviour of each bifurcation branch as the competition rate tends to infinity.

To present our main results clearly, we will introduce some notations that will be consistently used throughout the paper.
\begin{itemize}
	\item The ambient space is the radial function space $C_{r}^{2,\kappa}(\overline{B_1},\mathbb{R}^{2})$ for some $\kappa\in(0,1)$. More precisely,
	\begin{equation*}
C_{r}^{2,\kappa}(\overline{B_1},\mathbb{R}^{2}):=	\{u=(u_{1}, u_{2}): u_{i}\in C^{2,\kappa}(\overline{B_1}), u_{i}(x)=u_{i}(r), r=|x|, i=1,2\}.
	\end{equation*}
	\item $-\Delta_r$ denotes the Laplacian operator on radial function space with Neumann boundary condition.
	\item $\lambda_j$ denotes the $j$-th positive eigenvalue for $-\Delta_r$, $j=1, 2, \cdots$. 
	\item $f_j$ is the corresponding radial eigenfunction to $\lambda_j$. By \cite[Section 3.3]{Grebenkov-Nguyen-2013}, every  $\lambda_j$ is simple, and $f_j(r)$ has exactly $j$ simple zeros, where $r=|x|$.
\end{itemize}

%Our aim is to comprehensively analyze the bifurcation phenomenon surrounding the constant solution and to uncover non-constant radial solutions. Additionally, we seek to elucidate global bifurcation branches. Finally, in the case $\sigma=\mu$, we delve into investigating the limiting configuration on each branch as the competition rate goes to infinity.

The following theorem is related to the local bifurcation analysis and the instability around constant positive solutions.
 \begin{theorem}[Local bifurcation]\label{bifurcation-NBC}
 	Consider the following two-component competition system:
 	\begin{equation}\label{system2-comp}
 		\begin{cases}
 			-\Delta u_1=\mu u_1(1-u_1)-\beta \alpha u_1u_2,& \text{in}\ B_1\subset \mathbb{R}^N,\\
 			-\Delta u_2=\sigma u_2(1-u_2)-\beta \gamma u_1u_2,& \text{in}\ B_1\subset \mathbb{R}^N,\\
 			\frac{\partial u_1}{\partial n}=	\frac{\partial u_2}{\partial n}	=0,&\text{on}\ \partial B_1,
 		\end{cases}
 	\end{equation}
 	where $N\ge 2$, $\alpha>\gamma>0$, $\sigma\ge\mu>0$ and $\beta>\frac{\sigma}{\gamma}$. Then the following statements hold.
 	\begin{itemize}
 		\item[(i)] $(a_\beta,b_{\beta})=\frac{1}{\beta^2\alpha\gamma-\mu\sigma}\big((\beta\alpha-\mu)\sigma, (\beta\gamma-\sigma)\mu\big)$ is the only positive constant solution.
 		\item[(ii)] Given $\alpha, \gamma, \mu, \sigma$ such that  $\sqrt{\mu\sigma}\in (\lambda_k,\lambda_{k+1}]$,  then $\left((a_{\beta_{j}},b_{\beta_{j}}),\beta_{j}\right)$ is a bifurcation point for the $\beta_{j}$ satisfying
 		\begin{equation}\label{beta}
 			-\frac12\Big(\mu a_{\beta_{j}}+\sigma b_{\beta_{j}}-\sqrt{4\beta_{j}^{2}\alpha\gamma a_{\beta_{j}}b_{\beta_{j}}+(\mu a_{\beta_{j}}-\sigma b_{\beta_{j}})^{2}}\Big)=\lambda_j,\quad j=1, 2, \cdots, k.
 		\end{equation}
 		And there exist non-constant positive radial solutions in the solution curve. Moreover, the solution curve is of the form $\textbf{u}_{j}(s)=\frac{1}{\beta_{j}^2(s)\alpha\gamma-\mu\sigma}\big(\beta_{j}(s)\alpha\sigma-\mu\sigma, (\beta_{j}(s)\gamma\mu-\sigma\mu\big)+s\textbf{h}_{0,j}+\psi(s\textbf{h}_0,\beta_{j}(s))$, and the direction of bifurcation is 
 		\begin{equation*}
 			\textbf{h}_{0,j}=
 			f_{j}(r)\begin{pmatrix}
 				-\frac{\lambda_{j}+\sigma b_{\beta_{j}}}{\beta_{j}\gamma b_{\beta_{j}}}\\
 				1
 			\end{pmatrix}.
 		\end{equation*}
		\item[(iii)] The constant positive solution $(a_\beta,b_{\beta})$ is unstable.
 	\end{itemize}
 \end{theorem}
 \begin{remark}
 	$(a)$ In fact, $\frac{\sigma}{\gamma}<\beta_{1}<\beta_{2}<\cdots<\beta_{k-1}<\beta_{k}$ by Proposition \ref{mono-beta}. $(b)$ If $\sqrt{\mu\sigma}$ is sufficiently close to $\lambda_k$, then $\beta_k$ could be arbitrarily large. $(c)$  It holds that $\beta_j'(0)\neq 0$ for almost all $\sigma,\alpha,\gamma,\mu$ by Remark \ref{derivative-beta-neq0} and Lemma \ref{derivative}.
 \end{remark}
 
 Next, we discuss the global bifurcation phenomenon.
 \begin{theorem}[Global bifurcation]\label{global}
 	Under the assumptions in Theorem \ref{bifurcation-NBC}, then the $\left((a_{\beta_{j}},b_{\beta_{j}}),\beta_{j}\right)$ 
 		in Theorem \ref{bifurcation-NBC} $(ii)$ is actually also a global bifurcation point, $ j=1,2,\cdots, k$. %the following statements hold.
 	%\begin{itemize}
 	%		\item[(i)] 
 	%	The $\left((a_{\beta_{j}},b_{\beta_{j}}),\beta_{j}\right)$ 
 	%	in Theorem \ref{bifurcation-NBC} $(ii)$ is actually also a global bifurcation point, $ j=1,2,\cdots, k$. 
 	%	\item[(ii)] If additionally assume $\mu,\sigma \notin\sigma(-\Delta_r)$,  then. %and every global bifurcation branch is unbounded.
 	%\end{itemize}
 \end{theorem}
 
 If $\sigma=\mu\notin\sigma(-\Delta_r)$, then every global bifurcation branch is unbounded by Gui-Lou \cite{Gui-Lou-1994}. 
 In fact,	on each bifurcation branch, $\beta(s)$ is unbounded and goes to infinity. 
 	%Thus for $\beta$ sufficiently large, there are at least $k+1$ positive solutions.
% \end{remark}
Our final result is concentrated on the limiting profile of each branch under strong competition. 
\begin{theorem}[Limiting configuration]\label{limit-sys} 
	In the case $\sigma=\mu\notin\sigma(-\Delta_r)$, the following statements hold. 
	\begin{itemize}
	\item[(i)]
Every global bifurcation branch is unbounded. 
\item[(ii)]
 On the $j$-th bifurcation branch, the $w_{\beta}=\gamma u_{1,\beta}-\alpha u_{2,\beta}\rightarrow w$ strongly in $C_{r}^{2,\kappa}(B_{1})$  as $\beta$ tends to infinity, where $w$ satisfies
	\begin{equation*}
		-\Delta w= f(w),
		\mbox{with}\ 
		f(s)=
		\begin{cases}
			\mu s\left(1-\frac{s}{\gamma}\right),& s\ge0,\\
			\mu s\left(1+\frac{s}{\alpha}\right),& s\le0,
		\end{cases}
	\end{equation*}
	and $w$ has $j$ simple roots in $(0,1)$.
	\end{itemize}
\end{theorem}
 
 Our approach involves several key steps. In our analysis, we focus on $\beta$ as the bifurcation parameter.  Initially, we perform a detailed examination of local bifurcations. This involves deriving precise estimates of potential bifurcation locations. Additionally, we determine the direction of bifurcation, which provides insights into how system behaviour changes near constant solutions.  As a byproduct of our investigation, we establish the instability of constant solutions. We extend our analysis to investigate global bifurcation branches surrounding constant positive solutions. %A significant outcome is the demonstration that every bifurcation branch is unbounded. This result carries significant implications for understanding the long-term behaviour of the system.  
 Furthermore, in the case $\sigma=\mu$, we analyze the limiting configuration on each bifurcation branch as the competition rate $\beta\rightarrow+\infty$.

 The article is structured as follows. Section 2 is dedicated to exploring local bifurcation theory and the instability around the constant solution, with a focus on proving Theorem \ref{bifurcation-NBC}. Moving on to Section 3, our focus shifts towards global bifurcation theory and the proof of Theorem \ref{global}. Finally, in Section 4, in the case $\sigma=\mu$, we analyze the limiting profile on the global bifurcation branch and provide the proof of Theorem \ref{limit-sys}.  For the sake of brevity, we adopt a convention of not distinguishing between $\beta$ and $\beta_{j}$ in the proof unless ambiguity arises.

 \section{Proof of Theorem \ref{bifurcation-NBC}}

 \subsection{Local bifurcation}
 
Observe that 
there are totally four constant solutions $(0,0)$, $(1,0)$, $(0,1)$, $\frac{1}{\beta^2\alpha\gamma-\mu\sigma}\big((\beta\alpha-\mu)\sigma, (\beta\gamma-\sigma)\mu\big)$. 

\begin{proof}[Proof of Theorem \ref{bifurcation-NBC}(i)]
		It is obvious that 
	\begin{equation*}
	(a_{\beta},b_{\beta})=\frac{1}{\beta^2\alpha\gamma-\mu\sigma}\big((\beta\alpha-\mu)\sigma, (\beta\gamma-\sigma)\mu\big)
	\end{equation*}
	is the unique constant positive solution to \eqref{system2-comp} by the Cramer's rule.
\end{proof}

\begin{proof}[Proof of Theorem \ref{bifurcation-NBC}(ii)]
  For the constant positive solution $(a_{\beta},b_{\beta})=\frac{1}{\beta^2\alpha\gamma-\mu\sigma}\big((\beta\alpha-\mu)\sigma, (\beta\gamma-\sigma)\mu\big)$, consider the mapping $	\tilde{F}:C_{r}^{2,\kappa}(\overline{B_1},\mathbb{R}^{2})\times\mathbb{R}\rightarrow C_{r}^{0,\kappa}(\overline{B_1},\mathbb{R}^{2})$ 
 \begin{equation*}
 	\begin{aligned}
 		\tilde{F}((u_1,u_2),\beta)&:=\begin{pmatrix}
 			-\Delta u_1-\mu (u_1+a_{\beta})(1-u_1-a_{\beta})+\beta \alpha (u_1+a_{\beta}) (u_2+b_{\beta})\\
 			-\Delta u_2-\sigma (u_2+b_{\beta})(1-u_2-b_{\beta})+\beta \gamma (u_1+a_{\beta}) (u_2+b_{\beta})
 		\end{pmatrix}
 		\\
 		&=\begin{pmatrix}
 				-\Delta u_1+(u_1+a_{\beta})(\mu u_1+\beta \alpha u_2)\\
 			-\Delta u_2+(u_2+b_{\beta})(\sigma u_2+\beta \gamma u_1)
 		\end{pmatrix}\\
 		&=:L(\beta)(u_1,u_2)+P((u_1,u_2),\beta),
 	\end{aligned}
 \end{equation*}
 where 
 \begin{equation*}
 	L(\beta)(u_1,u_2)=\begin{pmatrix}
 		-\Delta u_1+a_{\beta}(\mu u_1+\beta \alpha u_2)\\
 		-\Delta u_2+b_{\beta}(\sigma u_2+\beta \gamma u_1)
 	\end{pmatrix},\  P((u_1,u_2),\beta)=\begin{pmatrix}
 	u_1(\mu u_1+\beta \alpha u_2)\\
 	u_2(\sigma u_2+\beta \gamma u_1)
 	\end{pmatrix}.
 \end{equation*}
 Then 
 $ \tilde{F}((0,0),\beta)=0$, $P((0,0),\beta)=0$ for all $\beta$. Moreover,
 \begin{equation*}
 	\begin{aligned}
 		P_x\big((u_1,u_2),\beta\big)(h_1,h_2)=\begin{pmatrix}
 			h_1(\mu u_1+\beta \alpha u_2)+u_1(\mu h_1+\beta \alpha h_2)\\
 			h_2(\sigma u_2+\beta \gamma u_1)+u_2(\sigma h_2+\beta \gamma h_1)
 		\end{pmatrix},
 	\end{aligned}
 \end{equation*}
 and 
 \begin{equation*}
 	\begin{aligned}
 		P_{x\beta}\big((u_1,u_2),\beta\big)(h_1,h_2)=\begin{pmatrix}
 			\alpha	(h_1 u_2+u_1  h_2)\\
 			\gamma( h_2 u_1+u_2 h_1)
 		\end{pmatrix}.
 	\end{aligned}
 \end{equation*}
 Therefore
 $P_x\big((0,0),\beta)\big)=0$, $P_{x\beta}\big((0,0),\beta)\big)=0$ for all $\beta$.
 
 \textbf{Step 1:} We try to find $\ker L(\beta)$ and $\text{Ran} L(\beta)$.
 In fact,
 \begin{equation*}
 	L(\beta)(h_1,h_2)
 	= \begin{pmatrix}
 		-\Delta h_1 \\
 		-\Delta h_2
 	\end{pmatrix}+A(\beta)
 	\begin{pmatrix}
 		h_1\\
 		h_2
 	\end{pmatrix},\ \text{where}\  A(\beta)=\begin{pmatrix}
 		\mu a_{\beta}&\beta\alpha a_{\beta}\\
 		\beta\gamma b_{\beta}&\sigma b_{\beta}
 	\end{pmatrix}.
 	\end{equation*}
 	For the matrix $A(\beta)$, there are two eigenvalues: 
 	\begin{equation}\label{eigen1}
 		\begin{aligned}
 			\delta_1(\beta)=\frac12\Big(\mu a_{\beta}+\sigma b_{\beta}-\sqrt{4\beta^{2}\alpha\gamma a_{\beta}b_{\beta}+(\mu a_{\beta}-\sigma b_{\beta})^{2}}\Big), \quad
 			\delta_2(\beta)=\frac12\Big(\mu a_{\beta}+\sigma b_{\beta}+\sqrt{4\beta^{2}\alpha\gamma a_{\beta}b_{\beta}+(\mu a_{\beta}-\sigma b_{\beta})^{2}}\Big).
 		\end{aligned}
 	\end{equation}
 	 Furthermore,
 	 \begin{small}
 	\begin{equation*}
 A(\beta)Q_\beta=Q_\beta\begin{pmatrix}
 	\delta_1(\beta)&0\\
 	0&\delta_2(\beta)
 \end{pmatrix},
 \ Q_\beta=\begin{pmatrix}
-\frac{-\mu a_{\beta}+\sigma b_{\beta}+\sqrt{4\beta^{2}\alpha\gamma a_{\beta}b_{\beta}+(\mu a_{\beta}-\sigma b_{\beta})^{2}}}{2\beta\gamma b_{\beta}}&-\frac{-\mu a_{\beta}+\sigma b_{\beta}-\sqrt{4\beta^{2}\alpha\gamma a_{\beta}b_{\beta}+(\mu a_{\beta}-\sigma b_{\beta})^{2}}}{2\beta\gamma b_{\beta}}\\
 	1&1
 \end{pmatrix}.
 \end{equation*}
 \end{small}
 Thus if $(h_1,h_2)\in \ker L(\beta)$, then 
 \begin{equation*}
 \begin{cases}
 	-\Delta v_1+\delta_1(\beta) v_1=0,&\text{in}\  B_1,\\
 	-\Delta v_2+\delta_2(\beta) v_2=0,&\text{in}\  B_1,
 \end{cases}\ \  \text{where} 
 \begin{pmatrix}
 	v_1\\
 	v_2
 \end{pmatrix}=Q_{\beta}^{-1}\begin{pmatrix}
 	h_1\\h_2
 \end{pmatrix}.
 \end{equation*}
 Thus $v_2\equiv0$.
If  $0<-\delta_1(\beta)=\lambda_j\in \sigma(-\Delta_r)$ is a simple eigenvalue, then  $v_1=f_{j}(r)$ and $f_{j}(r)$ has exactly $j$ simple roots.     In addition, $	\delta_1(\frac{\sigma}{\gamma})=0$ and
\begin{equation*}
	\begin{aligned}
		- \delta_1(\beta)&=\frac{\sqrt{\mu\sigma}}{2(\beta^2\alpha\gamma-\mu\sigma)}\Big[-\sqrt{\mu\sigma}[\beta(\alpha+\gamma)-(\mu+\sigma)]+\sqrt{4\beta^2\alpha\gamma(\beta\alpha-\mu)(\beta\gamma-\sigma)+\mu\sigma[\beta(\alpha-\gamma)-(\mu-\sigma)]^2}\Big]\\
		&\rightarrow \sqrt{\mu\sigma}\quad \text{as}\ \beta\rightarrow+\infty.
	\end{aligned}
\end{equation*}
{By Proposition \ref{mono-beta}, $-\delta_{1}(\beta)$ is increasing with respect to $\beta\in(\frac{\sigma}{\gamma},+\infty)$.} 
 Thus if $\sqrt{\mu\sigma}>\lambda_j$,
  there exists a unique $\beta_j$ such that $-\delta_1(\beta_j)=\lambda_{j}$ is a simple eigenvalue and 
 \begin{equation}\label{kernel-L-beta}
 \ker L(\beta_j)=\Big\{(h_1,h_2)|  \begin{pmatrix}
 	h_1\\
 	h_2
 \end{pmatrix}=Q_{\beta_j}\begin{pmatrix}
 	bf_{j}(r)\\0
 \end{pmatrix}=bf_{j}(r)\begin{pmatrix}
m(\beta_{j})\\
 	1
 \end{pmatrix}, \text{where}\ b\in\mathbb{R}\Big\},
 \end{equation}
 where $$m(\beta_{j}):= -\frac{-\mu a_{\beta_j}+\sigma b_{\beta_j}+\sqrt{4\beta_j^{2}\alpha\gamma a_{\beta_j}b_{\beta_j}+(\mu a_{\beta_j}-\sigma b_{\beta_j})^{2}}}{2\beta_j\gamma b_{\beta_j}}=-\frac{\lambda_{j}+\sigma b_{\beta_{j}}}{\beta_{j}\gamma b_{\beta_{j}}}<0,$$
 and thus $\dim \ker L(\beta_j)=1$. 
 
 On the other hand,  since 
 \begin{equation*}
 	[\text{Ran} L(\beta_j)]^{\bot}= \ker \big(L(\beta_j)\big)^*,
 \end{equation*}
 we try to find $ \ker \big(L(\beta_j)\big)^*$. Notice that
 \begin{equation*}
 	\begin{aligned}
 		\langle  L(\beta)\big(u_{1},u_{2}\big),(v_{1},v_{2})\rangle
 		= \langle\big(u_{1},u_{2}\big), (L(\beta))^{*}(v_{1},v_{2})\rangle,
 	\end{aligned}
 \end{equation*}
 where 
 \begin{equation*}
 	\begin{aligned}
 		(L(\beta))^{*}(v_{1},v_{2})
 		= \begin{pmatrix}
 			-\Delta v_1\\
 		-\Delta v_2
 		\end{pmatrix}+	M({\beta})
 		\begin{pmatrix}
 			v_1\\
 			v_2
 		\end{pmatrix}, \text{where} \  M({\beta}):= \begin{pmatrix}
 			\mu a_{\beta}&\beta\gamma b_\beta\\
 		\beta\alpha 	a_{\beta}&\sigma b_{\beta}
 		\end{pmatrix}.
 	\end{aligned}
 \end{equation*}
 For the matrix $M(\beta)$, there are two eigenvalues:
 	\begin{equation*}
 		\begin{aligned}
 			\delta_1(\beta)=\frac12\Big(\mu a_{\beta}+\sigma b_{\beta}-\sqrt{4\beta^{2}\alpha\gamma a_{\beta}b_{\beta}+(\mu a_{\beta}-\sigma b_{\beta})^{2}}\Big), \quad
 			\delta_2(\beta)=\frac12\Big(\mu a_{\beta}+\sigma b_{\beta}+\sqrt{4\beta^{2}\alpha\gamma a_{\beta}b_{\beta}+(\mu a_{\beta}-\sigma b_{\beta})^{2}}\Big).
 		\end{aligned}
 	\end{equation*}
 Furthermore,
 	\begin{equation*}
	\begin{split}
 		&M(\beta)P_{\beta}=P_{\beta}\begin{pmatrix}
 			\delta_1(\beta)&0\\
 			0&\delta_2(\beta)
 		\end{pmatrix}\\
 		 &P_{\beta}=\begin{pmatrix}
-\frac{-\mu a_{\beta}+\sigma b_{\beta}+\sqrt{4\beta^{2}\alpha\gamma a_{\beta}b_{\beta}+(\mu a_{\beta}-\sigma b_{\beta})^{2}}}{2\beta\alpha a_{\beta}}&-\frac{-\mu a_{\beta}+\sigma b_{\beta}-\sqrt{4\beta^{2}\alpha\gamma a_{\beta}b_{\beta}+(\mu a_{\beta}-\sigma b_{\beta})^{2}}}{2\beta\alpha a_{\beta}}\\
 	1&1
 \end{pmatrix}. 
 \end{split}	\end{equation*}
  Thus if $(v_1,v_2)\in \ker L(\beta)^*$, then 
 \begin{equation*}
 	\begin{cases}
 	-\Delta w_1+\delta_1(\beta) w_1=0,&\text{in}\  B_1,\\
 	-\Delta w_2+\delta_2(\beta) w_2=0,&\text{in}\  B_1,
 	\end{cases}\ \  \text{where} 
 	\begin{pmatrix}
 		w_1\\
 		w_2
 	\end{pmatrix}=P_{\beta}^{-1}\begin{pmatrix}
 		v_1\\v_2
 	\end{pmatrix}.
 \end{equation*}
 Thus $w_2=0$ and if 
 $-\delta_1(\beta)=\lambda_j\in \sigma(-\Delta)$ is a simple eigenvalue, then  $w_1=f_{j}(r)$.    
 Consequently, 
 \begin{equation}\label{range-L-beta}
 \begin{split}
 	[\text{Ran} L(\beta_j)]^{\bot}&=	\ker L(\beta_j)^*=\\&=\Big\{\textbf{v}_0=(v_1,v_2)|  \begin{pmatrix}
 		v_1\\
 		v_2
 	\end{pmatrix}=P_{\beta_j}\begin{pmatrix}
 		a f_{j}(r)\\0
 	\end{pmatrix}=af_{j}(r)\begin{pmatrix}
 	\frac{\gamma b_{\beta_{j}}}{\alpha a_{\beta_{j}}}m(\beta_{j})\\
 		1
 	\end{pmatrix}, \text{where} \ a\in\mathbb{R}\Big\}.
	\end{split}
 \end{equation}
 
 \textbf{Step 2:} We claim that $L_{\beta}(\beta_j)\textbf{h}_0\notin \text{Ran} L(\beta_j)$, where $\textbf{h}_0\in \ker L(\beta_j)$  and  \begin{equation*}
\textbf{h}_0=
 f_{j}(r)\begin{pmatrix}
 m(\beta_j)\\
 	1
 \end{pmatrix}.
 \end{equation*}
As a matter of fact, by direct calculations,
\begin{equation*}
	\begin{aligned}
			\frac{\mathrm{d}a_{\beta}}{\mathrm{d} \beta}&=-\frac{\sigma\alpha(\beta^2\alpha\gamma+\mu\sigma-2\beta\gamma\mu)}{(\beta^2\alpha\gamma-\mu\sigma)^2},
		\\
		\frac{\mathrm{d}b_{\beta}}{\mathrm{d} \beta}&=-\frac{\mu\gamma(\beta^2\alpha\gamma+\mu\sigma-2\beta\alpha\sigma)}{(\beta^2\alpha\gamma-\mu\sigma)^2}.
	\end{aligned}
\end{equation*}
Thus
\begin{small}
\begin{equation*}
	\begin{aligned}
			L_{\beta}(\beta)\textbf{h}&=\begin{pmatrix}
				\frac{\mathrm{d}a_{\beta}}{\mathrm{d} \beta}(\mu h_1+\beta\alpha h_2)+a_{\beta}\alpha h_2\\
					\frac{\mathrm{d}b_{\beta}}{\mathrm{d} \beta}(\beta\gamma h_1+\sigma h_2)+b_{\beta}\gamma h_1
			\end{pmatrix}=C(\beta)\begin{pmatrix}
		h_1\\h_2
			\end{pmatrix},\\
		 C(\beta)&=\frac{\mu\sigma}{(\beta^2\alpha\gamma-\mu\sigma)^2}\begin{pmatrix}
				-	\alpha(\beta^2\alpha\gamma+\mu\sigma-2\beta\mu\gamma)&\alpha(\beta^2\alpha\gamma+\mu\sigma-2\beta\alpha\sigma)\\
				\gamma(\beta^2\alpha\gamma+\mu\sigma-2\beta\mu\gamma)&-\gamma(\beta^2\alpha\gamma+\mu\sigma-2\beta\alpha\sigma)
			\end{pmatrix}.
	\end{aligned}
\end{equation*}
 \end{small}
 Furthermore,
 \begin{equation*}
 	\begin{aligned}
 		L_{\beta}(\beta_j)\textbf{h}_0&=C(\beta_j)f_{j}(r)\begin{pmatrix}
 m(\beta_{j})\\
 		1
 		\end{pmatrix}\\
 		&=f_{j}(r)\frac{\mu\sigma}{(\beta_j^2\alpha\gamma-\mu\sigma)^2}\begin{pmatrix}
 			-	\alpha(\beta_j^2\alpha\gamma+\mu\sigma-2\beta_j\mu\gamma)&\alpha(\beta_j^2\alpha\gamma+\mu\sigma-2\beta_j\alpha\sigma)\\
 			\gamma(\beta_j^2\alpha\gamma+\mu\sigma-2\beta_j\mu\gamma)&-\gamma(\beta_j^2\alpha\gamma+\mu\sigma-2\beta_j\alpha\sigma)
 		\end{pmatrix}\begin{pmatrix}
 			m(\beta_{j})\\
 			1
 		\end{pmatrix}\\
 		&=f_{j}(r)\frac{\mu\sigma}{(\beta_j^2\alpha\gamma-\mu\sigma)^2}\big[	\gamma(\beta_j^2\alpha\gamma+\mu\sigma-2\beta_j\mu\gamma)m(\beta_{j})-\gamma(\beta_j^2\alpha\gamma+\mu\sigma-2\beta_j\alpha\sigma)\big]\begin{pmatrix}
 			-\frac{\alpha}{\gamma}\\
 			1
 		\end{pmatrix}
		\\
		&=
		f_{j}(r)\frac{\mu\sigma\gamma}{(\beta_j^2\alpha\gamma-\mu\sigma)^2}\big[	(\beta_j^2\alpha\gamma+\mu\sigma-2\beta_j\mu\gamma)m(\beta_{j})-(\beta_j^2\alpha\gamma+\mu\sigma-2\beta_j\alpha\sigma)\big]\begin{pmatrix}
 			-\frac{\alpha}{\gamma}\\
 			1
 		\end{pmatrix}.
 	\end{aligned}
 \end{equation*}
{Since $(\beta_j^2\alpha\gamma+\mu\sigma-2\beta_j\mu\gamma)m(\beta_{j})-(\beta_j^2\alpha\gamma+\mu\sigma-2\beta_j\alpha\sigma)< 0$ by Lemma \ref{est-m-beta}.}
 It is sufficient to prove that $\langle L_{\beta}(\beta_j)\textbf{h}_0, \textbf{v}_0\rangle\neq 0$, where $\textbf{v}_0$ is defined in \eqref{range-L-beta}. Equivalently, we only need to prove $-\frac{\alpha}{\gamma}\cdot\frac{\gamma b_{\beta_{j}}}{\alpha a_{\beta_{j}}}m(\beta_{j})+1\neq 0$, which is obvious since $\beta_j>\frac{\sigma}{\gamma}$ and $m(\beta_{j})<0$. Hence the claim holds.
 
 \textbf{Step 3:} Applying \cite[Theorem I.5.1]{bifurcation-2012} or \cite[Theorem I.3.3]{chang2005}, there exists a nontrivial continuously differentiable curve through $\big((0,0),\beta_j\big)$
 \begin{equation}\label{curve}
 	\{(\textbf{u}(s), \beta_{j}(s))\big| s\in(-\delta,\delta),(\textbf{u}(0), \beta_{j}(0))=(\textbf{0},\beta_j) \},
 \end{equation}
 such that $\tilde{F}(\textbf{u}(s), \beta_{j}(s))=\textbf{0}$ for $s\in(-\delta,\delta)$, and all solutions of $\tilde{F}(\tilde{u},\beta)$ in a neighbourhood of $(\textbf{0},\beta_j)$ are on the trivial solution line or on the nontrivial curve \eqref{curve}. Moreover $\textbf{u}(s)=s\textbf{h}_0+\psi(s\textbf{h}_0,\beta_{j}(s))$, and the direction of bifurcation is 
 \begin{equation*}
 \textbf{u}'(0)=\textbf{h}_0=
 f_{j}(r)\begin{pmatrix}
 m(\beta_{j})\\
 	1
 \end{pmatrix}.
 \end{equation*}
 \end{proof}
 \begin{remark}\label{derivative-beta-neq0}
 We try to prove  $\beta_{j}'(0)\neq0$. By \cite[Section I.6]{bifurcation-2012}, it is sufficient to prove 
 \begin{equation*}
 	D_{xx}^2\tilde{F}(\textbf{0},\beta_j)[\textbf{h}_0,\textbf{h}_0]\notin\text{Ran} D_x\tilde{F}(\textbf{0},\beta_j)=\text{Ran} L(\beta_j).
 \end{equation*}
 In fact, 
 \begin{equation*}
 	D_{xx}^2\tilde{F}(\textbf{u},\beta)[\textbf{h},\textbf{v}]=\begin{pmatrix}
 		h_1(\mu v_1+\beta \alpha v_2)+v_1(\mu h_1+\beta \alpha h_2)\\
 		h_2(\sigma v_2+\beta \gamma v_1)+v_2(\sigma h_2+\beta \gamma h_1)
 	\end{pmatrix}.
 \end{equation*}
 Moreover, 
 \begin{equation*}
 	D_{xx}^2\tilde{F}(\textbf{0},\beta_j)[\textbf{h}_0,\textbf{h}_0]=\begin{pmatrix}
 		2h_1(\mu h_1+\beta_{j} \alpha  h_2)\\
 		2h_2(\sigma h_2+\beta_{j} \gamma h_1)
 	\end{pmatrix}=2f_{j}^2(r)\begin{pmatrix}
m(\beta_{j})[\mu m(\beta_{j})+\beta_{j}\alpha]\\
 \sigma+\beta_{j}\gamma m(\beta_{j})
 	\end{pmatrix}.
 \end{equation*}
 So we only need to prove $2f_{j}^2(r)\big( m(\beta_{j})[\mu m(\beta_{j})+\beta_{j}\alpha],
 \sigma+\beta_{j}\gamma m(\beta_{j})\big)\cdot \textbf{v}_0\neq 0$ for $\textbf{v}_0$  defined in \eqref{range-L-beta} . Equivalently, it suffices to prove 
 \begin{equation*}
 	m(\beta_{j})[\mu m(\beta_{j})+\beta_{j}\alpha]\cdot \frac{\gamma b_{\beta_{j}}}{\alpha a_{\beta_{j}}}m(\beta_{j})+\sigma+\beta_{j}\gamma m(\beta_{j})\neq 0,
 \end{equation*}
 {By Lemma \ref{derivative}, if $\sqrt{\mu\sigma}$ is sufficiently close to $\lambda_k$, then it is true for $j=k$.} Furthermore, if we make slight adjustments to the parameters $\alpha$, $\sigma$, $\mu$ and $\gamma$, then $\beta_{j}'(0)$ will not equal to zero for any $j$.
 \end{remark}
 
 \subsection{Instability at constant positive solutions}
 
 By \cite[p46]{shi2011} and \cite{Kishimoto-1985},  the stability of positive solution $(u_1,u_2)$ to \eqref{system2-comp} is determined by the eigenvalue problem:
 \begin{equation}\label{eigenvalue-stability}
 	\begin{cases}
 		-\Delta\xi=\lambda \xi+(\mu-2\mu u_1-\beta\alpha u_2)\xi-\beta\alpha u_1 \eta,&\text{in}\ B_1,\\
 		-\Delta\eta=\lambda \eta-\beta\gamma u_2 \xi+(\sigma-2\sigma u_2-\beta\gamma u_1)\eta,&\text{in}\ B_1,\\
 		\frac{\partial \xi}{\partial n}=	\frac{\partial \eta}{\partial n}=0,&\text{on}\ \partial B_1.
 	\end{cases}
 \end{equation}
 We consider the real principal eigenvalue $\lambda(u_1,u_2)$, which has the smallest real part among all the spectrum points. The solution $(u_1,u_2)$ is stable if the real part $\Re\lambda(u_1,u_2)>0$ and it is unstable otherwise.

 \begin{proof}[Proof of Theorem \ref{bifurcation-NBC}(iii)]	We consider $(a_\beta,b_\beta)=\frac{1}{\beta^2\alpha\gamma-\mu\sigma}\big((\beta\alpha-\mu)\sigma, (\beta\gamma-\sigma)\mu\big)$ in \eqref{eigenvalue-stability}, then it is reduced to 
 	\begin{equation*}
 		\begin{cases}
 			-\Delta\xi=\lambda \xi-\mu a_\beta\xi-\beta\alpha a_\beta \eta,&\text{in}\ B_1,\\
 			-\Delta\eta=\lambda \eta-\beta\gamma b_\beta \xi-\sigma b_\beta\eta,&\text{in}\ B_1,\\
 			\frac{\partial \xi}{\partial n}=	\frac{\partial \eta}{\partial n}=0,&\text{on}\ \partial B_1,
 		\end{cases}
 	\end{equation*}
 	Equivalently,
 	\begin{equation*}
 		(-\Delta-\lambda)\begin{pmatrix}
 			\xi \\ \eta
 		\end{pmatrix}=-A(\beta)\begin{pmatrix}
 			\xi\\ \eta
 		\end{pmatrix},\ \text{where}\   A(\beta)=\begin{pmatrix}
 			\mu a_\beta& \beta\alpha a_\beta\\
 			\beta\gamma b_\beta&\sigma b_\beta
 		\end{pmatrix}.
 	\end{equation*}
 	For the matrix $A(\beta)$, there are two eigenvalues: 
 	\begin{equation*}
 		\begin{aligned}
 			\delta_1(\beta)=\frac12\Big(\mu a_{\beta}+\sigma b_{\beta}-\sqrt{4\beta^{2}\alpha\gamma a_{\beta}b_{\beta}+(\mu a_{\beta}-\sigma b_{\beta})^{2}}\Big), \\
 			\delta_2(\beta)=\frac12\Big(\mu a_{\beta}+\sigma b_{\beta}+\sqrt{4\beta^{2}\alpha\gamma a_{\beta}b_{\beta}+(\mu a_{\beta}-\sigma b_{\beta})^{2}}\Big).
 		\end{aligned}
 	\end{equation*}
 	Thus 
 	\begin{equation*}
 		\begin{cases}
 			(	-\Delta-\lambda) v_1+\delta_{1}(\beta) v_1=0,&\text{in}\  B_1,\\
 			(-\Delta-\lambda) v_2+\delta_2(\beta) v_2=0,&\text{in}\  B_1,
 		\end{cases}\ \  \text{where} 
 		\begin{pmatrix}
 			v_1\\
 			v_2
 		\end{pmatrix}=Q_{\beta}^{-1}\begin{pmatrix}
 			\xi\\ \eta
 		\end{pmatrix}.
 	\end{equation*}
 	Thus in order that $(\xi,\eta)\neq (0,0)$, it is required that 
 	$\lambda-\delta_1(\beta)=\lambda_j\in \sigma(-\Delta)$. Notice that 	
	$-\delta_1(\beta)\rightarrow \sqrt{\mu\sigma}$ as  $\beta\rightarrow +\infty$. 
 	Thus the principal eigenvalue $\lambda=\delta_1(\beta)<0$, hence the solution $\frac{1}{\beta^2\alpha\gamma-\mu\sigma}\big((\beta\alpha-\mu)\sigma, (\beta\gamma-\sigma)\mu\big)$ is unstable. 
 \end{proof}

 \section{Global bifurcation} 
 \begin{lemma}\label{upperbound-ui}
 	Any positive solution of \eqref{system2-comp} is  such that $0<u_i<1$ in $\overline{B_1}$ for any $i=1, 2, 3$.
 \end{lemma}
 \begin{proof}
 	Take  $w_i=u_{i}-1$, then 
 	\begin{equation*}
 		\begin{cases}
 			(-\Delta+\mu)w_i <0,&\text{in}\ B_1,\\
 			\frac{\partial w_i}{\partial n}=0,&\text{on} \ \partial B_1.
 		\end{cases}
 	\end{equation*}
 		By  \cite[Maximum Principle 4.1]{amann-1976},  $w_i<0$ in $\overline{B_1}$.
 	Finally, we conclude that $0<u_i<1$  in $\overline{B_1}$.
 \end{proof}

 \begin{proof}[Proof of Theorem \ref{global}]
 We consider the following system:
 \begin{equation*}
 	\begin{cases}
 		(-\Delta+\mu) z_1=\mu u_1(2-u_1)-\beta \alpha u_1u_2,& \text{in}\ B_1,\\
 		(-\Delta+\sigma) z_2=\sigma u_2(2-u_2)-\beta \gamma u_1u_2,& \text{in}\ B_1,\\
 		\frac{\partial z_1}{\partial n}=	\frac{\partial z_2}{\partial n}=0,&\text{on}\ \partial B_1.
 	\end{cases}
 \end{equation*} 
 Define a map $F: C_{r}^{0,\kappa}(\overline{B_1},\mathbb{R}^{2})\times \mathbb{R}\rightarrow C_{r}^{0,\kappa}(\overline{B_1},\mathbb{R}^{2})$:
 \begin{equation*}
 	F\big((u_1,u_2),\beta\big)=\begin{pmatrix}
 		u_1\\u_2
 	\end{pmatrix}-	\begin{pmatrix}
 (-\Delta+\mu)^{-1}		\big[\mu u_1(2-u_1)-\beta \alpha u_1u_2\big]\\
 (-\Delta+\sigma)^{-1}	\big[	\sigma u_2(2-u_2)-\beta \gamma u_1u_2\big]
 	\end{pmatrix}.
 \end{equation*}
 Then $F\big((a_{\beta},b_{\beta}),\beta\big)=0$ for all $\beta$. Moreover, 
 define
 $\hat{F}: C(\overline{B_1},\mathbb{R}^2)\times \mathbb{R}\rightarrow C(\overline{B_1},\mathbb{R}^2)$:
 \begin{equation*}
 	\hat{F}\big((u_1,u_2),\beta\big)=F\big((u_1+a_{\beta},u_2+b_{\beta}),\beta\big)
 	=\begin{pmatrix}
 		u_1+a_{\beta}\\u_2+b_{\beta}
 	\end{pmatrix}-	\begin{pmatrix}
 (-\Delta+\mu)^{-1}	\big\{	(u_1+a_{\beta})	[\mu (1-u_1)-\beta \alpha u_2]\big\}\\
 (-\Delta+\sigma)^{-1}	\big\{	(u_2+b_{\beta})	[\sigma (1-u_2)-\beta \gamma u_1]\big\}
 	\end{pmatrix}.
 \end{equation*}
 Thus $	\hat{F}\big((0,0),\beta\big)=0$ for all $\beta$.
 To apply the global Rabinowitz bifurcation theorem \cite[Theorem II.3.3]{bifurcation-2012}, we need to calculate the $\text{ind}\Big(D_x \hat{F}\big((0,0),\beta\big), (0,0)\Big)$.
 In fact,
 \begin{equation*}
 	\begin{aligned}
 		D_x \hat{F}\big((u_1, u_2),\beta\big)(h_1,h_2)	 		=\begin{pmatrix}
 			h_1\\h_2
 		\end{pmatrix}-\begin{pmatrix}
 		(-\Delta+\mu)^{-1}\big\{[\mu(1-2u_2-a_{\beta})-\beta\alpha u_2]h_1-\beta\alpha(u_1+a_{\beta})h_2\big\}\\
 		(-\Delta+\sigma)^{-1}	\big\{-\beta\gamma(u_2+b_{\beta})h_1+[\sigma(1-2u_2-b_{\beta})-\beta\gamma u_1]h_2\big\}
 		\end{pmatrix}.
 	\end{aligned}
 \end{equation*}
 In particular,
 \begin{equation*}
 	\begin{aligned}
 		D_x \hat{F}\big((0, 0),\beta\big)(h_1,h_2)	
 		&=\begin{pmatrix}
 			h_1\\h_2
 		\end{pmatrix}-\begin{pmatrix}
 			(-\Delta+\mu)^{-1}\big\{\mu(1-a_{\beta})h_1-\beta\alpha a_{\beta}h_2\big\}\\
 		(-\Delta+\sigma)^{-1}	\big\{-\beta\gamma b_{\beta}h_1+\sigma(1-b_{\beta})h_2\big\}
 		\end{pmatrix}.
 	\end{aligned}
 \end{equation*}
 It is easy to see that $\ker D_x \hat{F}\big((0, 0),\beta_j\big)(h_1,h_2)	=\ker L(\beta_j)$ defined in \eqref{kernel-L-beta}.
 
 We try to calculate  $\text{ind}\Big(D_x \hat{F}\big((0,0),\beta\big), (0,0)\Big)$ for $\beta\neq \beta_j$ around $\beta_j$. In what follows, we make use of 
 the index formula given by \cite[Proposition 14.5]{ziedler1993}. 
 If $\lambda>1$ such that 
 \begin{equation*}
 	\begin{pmatrix}
 		(-\Delta+\mu)^{-1}\big\{\mu(1-a_{\beta})h_1-\beta\alpha a_{\beta}h_2\big\}\\
 		(-\Delta+\sigma)^{-1}	\big\{-\beta\gamma b_{\beta}h_1+\sigma(1-b_{\beta})h_2\big\}
 	\end{pmatrix}=\lambda\begin{pmatrix}
 		h_1\\h_2
 	\end{pmatrix}.
 \end{equation*}
 Equivalently,
 \begin{equation*}
 		-\Delta \begin{pmatrix}
 	h_1\\  h_2
 	\end{pmatrix}=D(\beta,\lambda)\begin{pmatrix}
 		h_1\\ h_2
 	\end{pmatrix}, \text{where}
 	\ D(\beta,\lambda)=\begin{pmatrix}
 		-\mu+\frac{\beta\alpha b_{\beta}}{\lambda}&\frac{-\beta\alpha a_{\beta}}{\lambda}\\
 		\frac{-	\beta\gamma b_{\beta}}{\lambda}&	-\sigma+\frac{\beta\gamma a_{\beta}}{\lambda}
 	\end{pmatrix}
	=\frac{1}{\lambda}\begin{pmatrix}
 	-\mu\lambda	+\beta\alpha b_{\beta}&-\beta\alpha a_{\beta}\\
 		-	\beta\gamma b_{\beta}&-\sigma\lambda+\beta\gamma a_{\beta}
 	\end{pmatrix}.
 \end{equation*}
 For $D(\beta,\lambda)$, there are two eigenvalues: 
 \begin{equation}\label{eigen-lambda}
 	\begin{aligned}
 	\delta_{1}(\beta,\lambda)&=\frac{1}{2\lambda}\Big[	-\mu\lambda	+\beta\alpha b_{\beta}-\sigma\lambda+\beta\gamma a_{\beta}-\sqrt{(\mu\lambda	-\beta\alpha b_{\beta}+\sigma\lambda-\beta\gamma a_{\beta})^2-4(\mu\sigma\lambda^2-\beta\alpha b_{\beta}\sigma\lambda-\beta\gamma a_{\beta}\mu\lambda)}\Big].
 	\\
 		\delta_{2}(\beta,\lambda)&=\frac{1}{2\lambda}\Big[	-\mu\lambda	+\beta\alpha b_{\beta}-\sigma\lambda+\beta\gamma a_{\beta}+\sqrt{(-\mu\lambda	+\beta\alpha b_{\beta}-\sigma\lambda+\beta\gamma a_{\beta})^2-4(\mu\sigma\lambda^2-\beta\alpha b_{\beta}\sigma\lambda-\beta\gamma a_{\beta}\mu\lambda)}\Big].
 	\end{aligned}
 \end{equation}
 
  Moreover, there exists $R_{\beta,\lambda}$ such that
 \begin{equation*}
 	D(\beta,\lambda) R_{\beta,\lambda}=R_{\beta,\lambda}\begin{pmatrix}
 			\delta_{1}(\beta,\lambda)&0\\
 		0&	\delta_{2}(\beta,\lambda)
 	\end{pmatrix}, \ \text{where}\ R_{\beta,\lambda}
	=\begin{pmatrix}
 		\frac{\mu\lambda-\beta\alpha b_{\beta}+\delta_{2,\lambda}(\beta)}{\beta\gamma b_{\beta}}&	-\frac{\sigma\lambda-\beta\gamma a_{\beta}+\delta_{2,\lambda}(\beta)}{\beta\gamma b_{\beta}}\\
 		1 &1
 	\end{pmatrix}.
 \end{equation*} 
 Therefore
 \begin{equation*}
 	\begin{cases}
 		-\Delta v_1=	\delta_{1}(\beta,\lambda) v_1,\\
 		-\Delta v_2=	\delta_{2}(\beta,\lambda) v_2,
 	\end{cases}\ \text{where}
 	\begin{pmatrix}
 		v_1\\v_2
 	\end{pmatrix}=R_{\beta,\lambda}^{-1}\begin{pmatrix}
 		h_1\\h_2
 	\end{pmatrix}.
 \end{equation*}
 Thus $v_1\equiv 0$ since $\delta_{1}(\beta,\lambda) <0$, and $v_2=f_{j}(r)$ if $\delta_{2}(\beta,\lambda)  =\lambda_j$. Indeed, by Proposition \ref{mono-lambda},
 \begin{equation*}
 \frac{\partial}{\partial \lambda}\delta_{2}(\beta,\lambda) <0,\quad   \frac{\partial}{\partial \beta}\delta_{2}(\beta,\lambda) >0.
 \end{equation*}
 And 
   $\delta_{2}(\beta_{j},1)=\lambda_j$, $\delta_{2}(\beta,\lambda)\rightarrow -\mu$ as $\lambda\rightarrow+\infty$. Consequently, for $\epsilon$ sufficiently small, we obtain $\delta_{2}(\beta,1)<\lambda_j$ in $(\beta_j-\epsilon,\beta_j)$ and $\delta_{2}(\beta,1)>\lambda_j$ in $(\beta_j,\beta_j+\epsilon)$. Thus  $	\text{ind}\Big(D_x \hat{F}\big((0,0),\beta\big), (0,0)\Big)$ jumps at $\beta=\beta_j$ from $+1$ to $-1$ or vice versa based on \cite[Proposition 14.5]{ziedler1993}.  It follows by \cite[(II.3.4), Theorem II.3.3]{bifurcation-2012} that $\big((0,0),\beta_j\big)$ is a global bifurcation point. Finally, let $\mathcal{S}$ be the closure of the set of nontrivial solutions of $\hat{F}\big((0,0),\beta\big)=0$ in $C(\overline{B}_1,\mathbb{R}^2)\times \mathbb{R}$. 
 Applying \cite[Theorem II.3.3, Theorem II.5.9]{bifurcation-2012}, it holds that $\big((0,0),\beta_j\big)\in \mathcal{S}$, and let $\mathcal{C}_j$ be the component of $\mathcal{S}$ to which $\big((0,0),\beta_j\big)$ belongs. Then either (i) $\mathcal{C}_j$
 is unbounded, or (ii) $\mathcal{C}_j$ contains some $\big((0,0),\beta_m\big)$ where $\beta_m\neq \beta_j$. 
 \end{proof}

 \section{Limiting configuration}
  \subsection{Proof of Theorem \ref{limit-sys}(i)} 
  Indeed, in the situation where  $\sigma=\mu\notin\sigma(-\Delta_r)$, Gui-Lou \cite{Gui-Lou-1994} have already demonstrated that every global bifurcation branch is unbounded. For the reader's convenience, we still present the detailed proof.  
 \begin{lemma}\label{isolated}
 	Under the assumptions $\alpha>\gamma>0$, $\beta>\frac{\mu}{\gamma}$ and $0<\mu \notin\sigma(-\Delta_r)$, the constant non-negative solutions $(0,0)$, $(0,1)$ and $(1,0)$ are isolated solutions, i.e., they do not appear on any bifurcation branch.
 \end{lemma}
 \begin{proof}
 	We consider
 	\begin{equation}\label{system-inverse2-comp}
 		\begin{cases}
 			(-\Delta+\mu+2\beta\alpha) z_1=\mu u_1(2-u_1)+2\beta\alpha u_{1}-\beta \alpha u_1u_2,& \text{in}\ B_1,\\
 			(-\Delta+\mu+2\beta\gamma) z_2=\mu u_2(2-u_2)+2\beta\gamma  u_{2}-\beta \gamma u_1u_2,& \text{in}\ B_1,\\
 			\frac{\partial z_1}{\partial n}=	\frac{\partial z_2}{\partial n}=0,&\text{on}\ \partial B_1,
 		\end{cases}
 	\end{equation}
 	 Define a map
 	\begin{equation*}
 		\Phi(u_{1}, u_{2})=(z_{1}, z_{2}) \Leftrightarrow (z_{1}, z_{2}) \ \text{solves}\  \eqref{system-inverse2-comp}.
 	\end{equation*}
 	It is easy to see that the fixed point of $\Phi$ is a solution to \eqref{system2-comp}. 
 	Suppose 
 	\begin{equation*}
 		(u_{1}, u_{2})\in B:=\{(u_{1}, u_{2}): u_{i}\in C(\overline{B_{1}}), 0\le u_{i}(x)\le 2\  \text{in}\   B_{1},   \  i=1,2\},
 	\end{equation*}
 	which is a subset of Banach space $C(\overline{B_1},\mathbb{R}^2)$ 
 	with the norm 
 	\begin{equation*}
 		\|(u_1, u_2)\|:=\max\{\|u_1\|_{C(\overline{B_1})}, \|u_2\|_{C(\overline{B_1})}\}.
 	\end{equation*}
 		Firstly, it is easy to see that the RHS of each equation must be non-negative and less than $\mu+4\beta\alpha$ or $\mu+4\beta\gamma$. By  \cite[Maximum Principle 4.1]{amann-1976}, we know \textcolor{black}{$z_{i}$ is a non-negative constant or  $z_{i}> 0$ in $\overline{B_{1}}$}. By \cite[Proposition 19.7]{resolvent-book}, the norm 
 		\begin{equation*}
 			\|(-\Delta+\mu+2\beta\alpha)^{-1}\|\le \frac{1}{\mu+2\beta\alpha},  \quad	\|(-\Delta+\mu+2\beta\gamma)^{-1}\|\le \frac{1}{\mu+2\beta\gamma}.
 		\end{equation*}
 		Hence $\|z_1\|_{C(\overline{B_{1}})}\le \frac{\mu+4\beta\alpha}{\mu+2\beta\alpha}<2$ and $\|z_2\|_{C(\overline{B_{1}})}\le \frac{\mu+4\beta\gamma}{\mu+2\beta\gamma}<2$.
 		Thus	$\Phi: B\rightarrow B$ is  compact.
 		Furthermore,
 		\begin{equation*}
 			\begin{aligned}
 				\Phi'\big((u_{1}, u_{2})\big)(h_{1},h_{2})
 				=\begin{pmatrix}
 					(-\Delta+\mu+2\beta\alpha)^{-1}	\{[2\mu(1-u_{1})+\beta\alpha(2-u_{2})]h_{1}-\beta\alpha u_{1}h_{2}\}\\
 					(-\Delta+\mu+2\beta\gamma)^{-1}	\{-\beta\gamma u_{2}h_{1}+ [2\mu(1-u_{2})+\beta\gamma(2-u_{1})]h_{2}\}
 				\end{pmatrix}.
 			\end{aligned}
 		\end{equation*}
 		
 		\begin{itemize}
 			\item[(a)] 
 			In order that  $(0,0)$ is isolated,  it suffices to prove $1$ is not the eigenvalue of $\Phi'\big((0, 0)\big)$. Indeed,
 			\begin{equation*}
 				\begin{aligned}
 					&\quad\Phi'\big((0, 0)\big)(h_{1},h_{2})=\begin{pmatrix}
 						(-\Delta+\mu+2\beta\alpha)^{-1}	\{(2\mu+2\beta\alpha) h_{1}\}\\
 						(-\Delta+\mu+2\beta\gamma)^{-1}	\{ (2\mu+2\beta\gamma)h_{2}\}
 					\end{pmatrix}.
 				\end{aligned}
 			\end{equation*}
 			Notice that $\Phi'\big((0, 0)\big)(h_{1},h_{2})=(h_1, h_2)$ if and only if
 			$-\Delta h_1=\mu h_1$ or $-\Delta h_{2}=\mu h_{2}$. If  { $\mu \notin\sigma(-\Delta_{r})$}, it follows that $I-\Phi'\big((0,  0)\big)$ is invertible and  $(0,0)$ is an isolated fixed point.

 			\item[(b)] To show that  $(1,0)$ is isolated,  it is sufficient to prove $1$ is not the eigenvalue of $\Phi'\big((1, 0)\big)$.
 			\begin{equation*}
 				\begin{aligned}
 					&\quad\Phi'\big((1, 0)\big)(h_{1},h_{2})=\begin{pmatrix}
 						(-\Delta+\mu+2\beta\alpha)^{-1}	\{2\beta\alpha h_{1}-\beta\alpha h_2\}\\
 						(-\Delta+\mu+2\beta\gamma)^{-1}	\{ (2\mu+\beta\gamma) h_{2}\}
 					\end{pmatrix}.
 				\end{aligned}
 			\end{equation*}
 			Notice that $\Phi'\big((1, 0)\big)(h_{1},h_{2})=(h_1, h_2)$ if and only if
 			\begin{equation*}
 				(-\Delta+\mu)h_1=-\beta\alpha h_2,\quad
 				(-\Delta+\beta\gamma-\mu)h_2=0,
 			\end{equation*}
 			then $h_2=h_1\equiv0$. Thus $I-\Phi'\big((1, 0)\big)$ is invertible and  $(1,0)$ is an isolated fixed point.

 			\item[(c)] We try to prove $(0,1)$ is isolated. Equivalently, we need to prove $1$ is not the eigenvalue of $\Phi'\big((0, 1)\big)$.
 			\begin{equation*}
 				\begin{aligned}
 					&\quad\Phi'\big((0,1)\big)(h_{1},h_{2})=\begin{pmatrix}
 						(-\Delta+\mu+2\beta\alpha)^{-1}	\{(2\mu+\beta\alpha)h_1\}\\
 						(-\Delta+\mu+2\beta\gamma)^{-1}	\{-\beta\gamma h_1+ 2\beta\gamma h_{2}\}
 					\end{pmatrix}.
 				\end{aligned}
 			\end{equation*}
 			Notice that $\Phi'\big((1, 0)\big)(h_{1},h_{2})=(h_1, h_2)$ if and only if
 			\begin{equation*}
 				(-\Delta+\beta\alpha-\mu)h_1=0,\quad
 				(-\Delta+\mu)h_2=-\beta\gamma h_1,
 			\end{equation*}
 			then $h_2=h_1\equiv0$.
 			Thus $(0,1)$ is an isolated fixed point of $\Phi$.  
 		\end{itemize}
 	\end{proof}
 	
 	\begin{lemma}\label{positive-branch}
 		On each bifurcation branch, the solutions are always positive.
 	\end{lemma}
 	\begin{proof}
 		From section 2, we know that the bifurcation branch is locally of the form 
 		\begin{equation*}
 			u_1(s)=\frac{\mu}{\beta^2(s)\alpha\gamma-\mu^2}(\beta(s)\alpha-\mu)-\frac{\alpha}{\gamma}sf_j(r)+o(s), u_2(s)=\frac{\mu}{\beta^2(s)\alpha\gamma-\mu^2}(\beta(s)\gamma-\mu)+sf_j(r)+o(s).
 		\end{equation*}
 		Thus $u_i(s)$ is positive for $s$ small. Furthermore, going back to the global bifurcation branch,  as $s$ moves away from $0$,  if for instance, $u_1(s_0)$ firstly touches $0$ for some $s_0$, then by the maximum principle, $u_1(s_0)\equiv 0$, which means $u_2(s_0)\equiv 0$ or $u_2(s_0)\equiv 1$, a contradiction by Lemma \ref{isolated}.
 	\end{proof}

Set $(w_1, w_2):=(u_1-a_\beta, u_2-b_{\beta})$, then $(w_1,w_2)$ solves
 \begin{equation}\label{w-system2-comp}
	\begin{cases}
		-\Delta w_1= (w_1+a_\beta)(-\mu w_1-\beta \alpha w_2),& \text{in}\ B_1,\\
		-\Delta w_2=(w_2+b_\beta)(-\mu w_2-\beta \gamma w_1),& \text{in}\ B_1,\\
		\frac{\partial w_1}{\partial n}=	\frac{\partial w_2}{\partial n}	=0,&\text{on}\ \partial B_1,
	\end{cases}
\end{equation}
 where $\mu, \beta>0$, $\alpha>\gamma>0$ and $\beta>\frac{\sigma}{\gamma}$. There are two trivial solutions $(0,0)$ and $(-a_{\beta},-b_{\beta})$.

 A solution of \eqref{w-system2-comp} is said to be locked if $w_{1}/w_{2}$ is constant. We first describe the set of locked solutions.

\begin{lemma}\label{w-eq-2}
		Let $w(x)$ be the solution to 
	\begin{equation}\label{w-equation}
		-\Delta w=-\mu w-w^{2}\ \text{in}\ B_1,\quad 	\frac{\partial w}{\partial n}=0 \  \text{on}\ \partial B_1.
	\end{equation}
	Then all the locked solution of \eqref{w-system2-comp} is of the form  $w_1=\frac{\beta\alpha-\mu}{\beta^2\alpha\gamma-\mu^2}w(x),  w_2=\frac{\beta\gamma-\mu}{\beta^2\alpha\gamma-\mu^2}w(x)$.
\end{lemma}
\begin{proof}
	The statement is true by direct calculations.
\end{proof}
\begin{remark}\label{w-equation2}
	By  \cite[Maximum Principle 4.1]{amann-1976}, it holds that $w\equiv 0$ or $w\equiv -\mu$ or $w<0$ in $\overline{B}_1$. In particular, by integrating on both sides of \eqref{w-equation}, we have
	\begin{equation*}
		0=\int_{B_1}w(w+\mu)\mathrm{d}x,
	\end{equation*}
	thus if $w$ is not a constant solution, then $w+\mu$ changes sign in $B_1$.
\end{remark}

\begin{lemma}
	Assume $(w_1,w_2)$ is a radial solution  to \eqref{w-system2-comp}, and set $v=(\beta\gamma-\mu)w_1-(\beta\alpha-\mu)w_2$. Then either $v\equiv 0$ or   $v(r)$ has only simple roots in $(0,1)$, where $r=|x|$.
\end{lemma}
\begin{proof}
	For $v=(\beta\gamma-\mu)w_1-(\beta\alpha-\mu)w_2$,  it holds that
	\begin{equation}\label{v-eq}
		\begin{cases}
			-\Delta v=\mu\Big[\frac{(\beta\alpha-\mu)(\beta\gamma-\mu)}{\beta^2\alpha\gamma-\mu}-(w_1+w_2)\Big]v,& \text{in}\ B_1,\\
			\frac{\partial v}{\partial n}	=0,&\text{on}\ \partial B_1.
		\end{cases}
	\end{equation}
	Furthermore, since $v(x)=v(r)$, where $r=|x|$, it follows from \eqref{v-eq} that
	\begin{equation*}
		\begin{cases}
			-[r^{N-1}v'(r)]'=\mu r^{N-1}\Big[\frac{(\beta\alpha-\mu)(\beta\gamma-\mu)}{\beta^2\alpha\gamma-\mu}-w_1(r)-w_2(r)\Big]v(r),& \text{for}\ r\in(0,1),\\
			v'(0)=v'(1)=0.
		\end{cases}
	\end{equation*}
	Therefore the zeros of $v(r)$ cannot have a cluster point in $(0,1)$ by \cite[p326]{Hartman-2002} if $v(r)\not\equiv 0$.	
\end{proof}

Define 
\begin{equation*}
	S_j=\{v\in C^1(0,1)|  v  \mbox{ has exactly } j\mbox{ simple roots in } (0,1), v(0)\neq 0, v(1)\neq 0\}.
\end{equation*}
Then  $S_j$ is an open set in $C^1(0,1)$ and $S_i\cap S_j=\emptyset$ for $i\neq j$.

\begin{lemma}\label{local-sj}
	There exists a neighborhood $O_j$ of $\left((0,0),\beta_j\right)$ such that if $\left((w_1,w_2),\beta\right)\in O_j\setminus\{(0,0)\times \mathbb{R}\}$ is a solution to \eqref{w-system2-comp}, then $v=(\beta\gamma-\mu)w_1-(\beta\alpha-\mu)w_2\in S_j$.
\end{lemma}
\begin{proof}
From the analysis in section 2, in a small neighbourhood of $\left((0,0),\beta_j\right)$, the solution $(w_1, w_2)$ is of the form
\begin{equation*}
	w_1(s)=-\frac{\alpha}{\gamma}sf_j(r)+o(s),\quad w_2(s)=sf_j(r)+o(|s|),\quad s\in(-\delta,\delta).
\end{equation*}
Thus 
\begin{equation*}
	v(s)=(\beta(s)\gamma-\mu)w_1-(\beta(s)\alpha-\mu)w_2=\frac{\mu(\alpha+\gamma)-2\beta(s)\alpha}{\gamma}sf_j(r)+o(|s|).
\end{equation*}
Notice that $f_j(r)$ has exactly $j$ simple roots in $(0,1)$, therefore the conclusion holds.
\end{proof}
 
\begin{theorem}\label{unbdd-branch}
	For any $\left((w_1,w_2),\beta\right)\notin \{(0,0)\times\mathbb{R}\}$ lying on the branch of nontrivial solutions $\mathcal{C}_j$ passing through $\left((0,0),\beta_j\right)$,  it holds that $v=(\beta\gamma-\mu)w_1-(\beta\alpha-\mu)w_2\in S_j$. Consequently, the bifurcation branch is unbounded.
\end{theorem}
\begin{proof}
	By Lemma \ref{local-sj}, the conclusion holds locally. We shall prove it globally. If not, there exists nontrivial solution $\left((w_1,w_2),\beta\right)$ lying on $\mathcal{C}_j$, and $v\in \partial S_j$. Then $v$ has at least a double zero, i.e., there exists $t\in[0,1]$ such that $v(t)=v'(t)=0$, it follows from the uniqueness of ODE that $v\equiv0$. Thus $(\beta\gamma-\mu)w_1=(\beta\alpha-\mu)w_2$, i.e., $(w_1,w_2)$ is a locked solution. By Lemma \ref{positive-branch}, Lemma \ref{w-eq-2} and Remark \ref{w-equation2}, we obtain $w_1(x)\equiv w_2(x)\equiv 0$. Therefore,  $\beta=\beta_m$ and $m\neq j$, while which contradicts with Lemma \ref{local-sj}. Hence the conclusion holds globally.
\end{proof}

\subsection{Proof of Theorem \ref{limit-sys}(ii)}
In light of Lemma \ref{upperbound-ui}, following the arguments in \cite{CTV-2005} and \cite{NTTV-2010}, we know that $u_{i,\beta}$ has a uniform $C^{0,\kappa}$ bound for any $\kappa\in(0,1)$.
In view of Lemma \ref{positive-branch} and Theorem \ref{unbdd-branch}, we know that on each bifurcation branch, $\beta(s)$ is unbounded and goes to infinity. 
Thus for $\beta$ sufficiently large, there are at least $k$ positive solutions.

\begin{proof}[Proof of Theorem \ref{limit-sys}(ii)]
Recall that
	\begin{equation*}
	\begin{cases}
		-\Delta u_{1,\beta}=\mu u_{1,\beta}(1-u_{1,\beta})-\beta \alpha u_{1,\beta}u_{2,\beta},& \text{in}\ B_1,\\
		-\Delta u_{2,\beta}=\mu u_{2,\beta}(1-u_{2,\beta})-\beta \gamma u_{1,\beta}u_{2,\beta},& \text{in}\ B_1,\\
		\frac{\partial u_{1,\beta}}{\partial n}=	\frac{\partial u_{2,\beta}}{\partial n}	=0,&\text{on}\ \partial B_1.
	\end{cases}
\end{equation*}
\textbf{Claim.} There exists a positive constant $C$, independent of $\beta$, such that 
\begin{equation*}
	\|u_{i,\beta}\|_{H^1(B_1)}+\|u_{i,\beta}\|_{L^{\infty}(B_1)}\le C,\quad i=1,2.
\end{equation*}
In fact, by Lemma \ref{upperbound-ui}, it holds that $\|u_{i,\beta}\|_{L^{\infty}(B_1)}\le 1$, $i=1,2$.
Furthermore,
 by choosing $u_{i,\beta}$ as the test function for the equation of $u_{i,\beta}$, one has
\begin{equation}\label{int}
	\begin{aligned}
	\int_{B_1}|\nabla u_{1,\beta}|^2\mathrm{d}x		&=	\int_{B_1} \mu u^2_{1,\beta}(1-u_{1,\beta})\mathrm{d}x-\int_{B_1}\beta \alpha u^2_{1,\beta}u_{2,\beta}\mathrm{d}x<\int_{B_1} \mu u^2_{1,\beta}(1-u_{1,\beta})\mathrm{d}x<\mu|B_1|,\\
	\int_{B_1}|\nabla u_{2,\beta}|^2\mathrm{d}x	&=		\int_{B_1}\mu u^2_{2,\beta}(1-u_{2,\beta})\mathrm{d}x-	\int_{B_1}\beta \gamma u_{1,\beta}u^2_{2,\beta}<\int_{B_1}\mu u^2_{2,\beta}(1-u_{2,\beta})\mathrm{d}x<\mu|B_1|.
	\end{aligned}
\end{equation}
Hence the claim holds.  Similar to the arguments in \cite[Theorem 2.2]{CTV-2005}, there exists $u_{i}\in H^1(B_1)$ such that $u_{i,\beta}\rightharpoonup u_i$ in $H^1(B_1)$ weakly, $i=1, 2$. Thus $u_{i,\beta}\rightarrow u_i$ a.e. in $B_1$ and  $\|u_i\|_{L^{\infty}(B_1)}\le 1$ holds by the Sobolev embedding and the claim. Moreover, from the above integration, it is easy to see that
\begin{equation*}
	\int_{B_1}u^2_{1,\beta}u_{2,\beta}\mathrm{d}x\rightarrow0, \quad \text{as}\ \beta\rightarrow\infty,
\end{equation*}
which implies that $u_{1,\beta}\cdot u_{2,\beta}\rightarrow 0$ a.e. in $B_1$ and accordingly $u_i\cdot u_j\rightarrow 0$ a.e. in $B_1$.

For $w_\beta=\gamma u_{1,\beta}-\alpha u_{2,\beta}$,  it holds that
\begin{equation*}
	\begin{cases}
			-\Delta w_\beta=\mu \left(w_\beta-\gamma u^2_{1,\beta}+\alpha u^2_{2,\beta}\right),& \text{in}\ B_1,
	\\
		\frac{\partial w_\beta}{\partial n}	=0,&\text{on}\ \partial B_1.
	\end{cases}
\end{equation*}
As $\beta$ goes to infinity, we see that $w_\beta\rightharpoonup w:= \gamma u_1-\alpha u_2$ weakly in $H^1$ and $w=\gamma u_1$ on $\{u_1>0\}$. Therefore $-\Delta w=-\Delta(\gamma u_1)$ on $\{u_1>0\}$. In fact, 
\begin{equation*}
		\begin{cases}
		-\Delta w=\mu \left(w-\gamma u^2_{1}+\alpha u^2_{2}\right),& \text{in}\ B_1,
		\\
		\frac{\partial w}{\partial n}	=0,&\text{on}\ \partial B_1.
	\end{cases}
\end{equation*}

In addition, to get the strong convergence,  by testing the equation of $w$ by $u_1$, we get
\begin{equation*}
	\int_{B_1}|\nabla u_1|^2\mathrm{d}x=\int_{B_1}\mu (1-u_1)u_1^2\mathrm{d}x.
\end{equation*}
In light of \eqref{int}, we see that 
\begin{equation*}
		\int_{B_1}|\nabla u_1|^2\mathrm{d}x\ge \limsup 	\int_{B_1}|\nabla u_{1,\beta}|^2\mathrm{d}x,
\end{equation*}
thus  $u_{1,\beta}\rightarrow u_1$ strongly in $H^1(B_1)$. Similarly,  $u_{2,\beta}\rightarrow u_2$ strongly in $H^1(B_1)$. Moreover, since $u_{i,\beta}\rightarrow u_{i}$ in $C_{r}^{0,\kappa}$, by the elliptic regularity theory, we see that $w_\beta\rightarrow w$ in $C_{r}^{2,\kappa}$.

Furthermore, notice that on the $j$-th global bifurcation branch, from  Theorem \ref{unbdd-branch}, $v_{\beta}=(\beta\gamma-\mu)u_{1,\beta}-(\beta\alpha-\mu)u_{2,\beta}$ has exactly $j$ simple roots in $(0,1)$. Consequently, 
\begin{equation*}
\frac{1}{\beta}	v_{\beta}=w_{\beta}-\frac{\mu}{\beta}u_{1,\beta}+\frac{\mu}{\beta}u_{2,\beta}\rightarrow w,
\end{equation*}
thus $w$ has exactly $j$ simple roots.
\end{proof}

\section{Appendix}
We give some notations below.
\begin{equation*}
	\begin{aligned}
			f(\beta)&:=4\beta^2\alpha\gamma(\beta\alpha-\mu)(\beta\gamma-\sigma)+\mu\sigma[\beta(\alpha-\gamma)-(\mu-\sigma)]^2
						\\
			&
			\ =4\beta^4\alpha^2\gamma^2-4\beta^3\alpha\gamma(\alpha\sigma+\gamma\mu)+\mu\sigma\big[\beta^2(\alpha+\gamma)^2-2\beta(\alpha-\gamma)(\mu-\sigma)+(\mu-\sigma)^2\big].
			\\
		h(\beta)&:=f'(\beta)(\beta^2\alpha\gamma-\mu\sigma)-4\beta\alpha\gamma{f(\beta)}.	
	\end{aligned}
\end{equation*}
\subsection{Monotonicity of  $-\delta_{1}(\beta)$ in local bifurcation analysis}
\begin{prop}\label{mono-beta}
	Assume $\sigma\ge \mu$, then $-\delta_{1}(\beta)$ defined in \eqref{eigen1}  is increasing with respect to $\beta \in (\frac{\sigma}{\gamma},+\infty)$.
\end{prop}
\begin{proof}
We try to check the monotonicity of $-\delta_{1}(\beta)$ with respect to $\beta$. Notice that
\begin{small}
	\begin{equation*}
		\begin{aligned}
			- \delta_1(\beta)&=\frac{\sqrt{\mu\sigma}}{2(\beta^2\alpha\gamma-\mu\sigma)}\Big[-\sqrt{\mu\sigma}[\beta(\alpha+\gamma)-(\mu+\sigma)]+\sqrt{4\beta^2\alpha\gamma(\beta\alpha-\mu)(\beta\gamma-\sigma)+\mu\sigma[\beta(\alpha-\gamma)-(\mu-\sigma)]^2}\Big]
		\\
		&	=\frac{\sqrt{\mu\sigma}}{2}\frac{1}{\beta^2\alpha\gamma-\mu\sigma}\Big[-\beta \sqrt{\mu\sigma}(\alpha+\gamma)+\sqrt{\mu\sigma}(\mu+\sigma)+\sqrt{f(\beta)}\Big].
		\end{aligned}
	\end{equation*}
\end{small}
Therefore,
\begin{small}
	\begin{equation*}
		\begin{aligned}
			-	\frac{\mathrm{d}}{\mathrm{d}\beta}\delta_{1}(\beta)
			&=\frac{\sqrt{\mu\sigma}}{2}\frac{1}{2(\beta^2\alpha\gamma-\mu\sigma)^2\sqrt{f(\beta)}}g(\beta),
		\end{aligned}
	\end{equation*}
\end{small}
where 
\begin{equation*}
	g(\beta)=f'(\beta)(\beta^2\alpha\gamma-\mu\sigma)-4\beta\alpha\gamma{f(\beta)}+2\sqrt{f(\beta)} \sqrt{\mu\sigma}\big\{	(\beta^2\alpha\gamma+\mu\sigma) (\alpha+\gamma)-2\beta\alpha\gamma(\mu+\sigma)\big\}.
\end{equation*}
{It is sufficient to prove $g(\beta)>0$ for $\beta\in(\frac{\sigma}{\gamma},+\infty)$. } In fact,
\begin{equation*}
	\begin{aligned}
		g(\frac{\sigma}{\gamma})
		&=4\sigma\gamma^{-1}(\sigma\alpha-\gamma\mu)^2(\beta^2\alpha\gamma-\mu\sigma)>0.
	\end{aligned}
\end{equation*}
And
\begin{equation*}
\begin{aligned}
(\beta^2\alpha\gamma+\mu\sigma) (\alpha+\gamma)-2\beta\alpha\gamma(\mu+\sigma)
&\ge (\sigma^{2}\alpha\gamma^{-1}+\mu\sigma) (\alpha+\gamma)-2\sigma\alpha(\mu+\sigma)
\\
&\ge 
\frac{\sigma}{\gamma}(\sigma\alpha-\gamma\mu)(\alpha-\gamma)>0.
\end{aligned}
\end{equation*}
Therefore one only needs to prove $h(\beta)=f'(\beta)(\beta^2\alpha\gamma-\mu\sigma)-4\beta\alpha\gamma{f(\beta)}>0$. Indeed,
\begin{equation*}
	\begin{aligned}
		h(\beta)=&f'(\beta)(\beta^2\alpha\gamma-\mu\sigma)-4\beta\alpha\gamma{f(\beta)}
	\\
	=&\big\{16\beta^3\alpha^2\gamma^2-12\beta^2\alpha\gamma(\alpha\sigma+\gamma\mu)+2\mu\sigma\beta(\alpha+\gamma)^2-2\mu\sigma(\alpha-\gamma)(\mu-\sigma)\big\}(\beta^2\alpha\gamma-\mu\sigma)
		\\
		& -4\beta\alpha\gamma\big\{4\beta^4\alpha^2\gamma^2-4\beta^3\alpha\gamma(\alpha\sigma+\gamma\mu)+\mu\sigma\beta^2(\alpha+\gamma)^2-2\beta\mu\sigma(\alpha-\gamma)(\mu-\sigma)+\mu\sigma(\mu-\sigma)^2\big] \big\}
		\\
		=&4\beta^4\alpha^2\gamma^2(\alpha\sigma+\gamma\mu)-2\beta^3\alpha\gamma\mu\sigma(\alpha^2+\gamma^2+10\alpha\gamma)+6\beta^2\alpha\gamma\mu\sigma(\alpha+\gamma)(\mu+\sigma)\\
		&-2\beta\mu\sigma\big[\mu\sigma(\alpha+\gamma)^2+2\alpha\gamma(\mu-\sigma)^2\big]+2\mu^2\sigma^2(\alpha-\gamma)(\mu-\sigma).
	\end{aligned}
\end{equation*}
Note that 
\begin{equation*}
	\begin{aligned}
	h(\frac{\sigma}{\gamma})=	f'(\frac{\sigma}{\gamma})(\sigma^2\alpha\gamma^{-1}-\mu\sigma)-4\sigma\alpha{f(\frac{\sigma}{\gamma})}=2\sigma^2\gamma^{-2}(\sigma\alpha-\gamma\mu)^2(2\sigma\alpha-\mu\alpha-\mu\gamma){\ge0},
	\end{aligned}
\end{equation*}
{since $2\sigma \alpha\ge\mu(\alpha+\gamma)$.}  Accordingly,
\begin{equation*}
	\begin{aligned}
		h'(\beta)=&16\beta^3\alpha^2\gamma^2(\alpha\sigma+\gamma\mu)-6\beta^2\alpha\gamma\mu\sigma(\alpha^2+\gamma^2+10\alpha\gamma)+12\beta\alpha\gamma\mu\sigma(\alpha+\gamma)(\mu+\sigma)\\
		&-2\mu\sigma\big[\mu\sigma(\alpha+\gamma)^2+2\alpha\gamma(\mu-\sigma)^2\big].
		\\
		h''(\beta)=&48\beta^2\alpha^2\gamma^2(\alpha\sigma+\gamma\mu)-12\beta\alpha\gamma\mu\sigma(\alpha^2+\gamma^2+10\alpha\gamma)+12\alpha\gamma\mu\sigma(\alpha+\gamma)(\mu+\sigma)
		\\
		=&12\alpha\gamma\big\{4\beta^2\alpha\gamma(\alpha\sigma+\gamma\mu)-\beta\mu\sigma(\alpha^2+\gamma^2+10\alpha\gamma)+\mu\sigma(\alpha+\gamma)(\mu+\sigma)\big\}.
		\\
		h'''(\beta)
		\ge	h'''(\frac{\sigma}{\gamma})
		=	&12\alpha\gamma\sigma\big[8\alpha(\alpha\sigma+\gamma\mu)-\mu(\alpha^2+\gamma^2+10\alpha\gamma)\big]
		\\
		=&12\alpha\gamma\sigma\big[8\alpha^2\sigma-\mu(\alpha+\gamma)^2\big]
		\ge12\alpha\gamma\sigma(4\sigma\alpha-\mu\alpha-\mu\gamma){>0}.
	\end{aligned}
\end{equation*}
Thus
\begin{equation*}
	\begin{aligned}
	h''(\beta)\ge	h''(\frac{\sigma}{\gamma})&=12\alpha\gamma\big\{4\sigma^2\alpha\gamma^{-1}(\alpha\sigma+\gamma\mu)-\mu\sigma^2\gamma^{-1}(\alpha^2+\gamma^2+10\alpha\gamma)+\mu\sigma(\alpha+\gamma)(\mu+\sigma)\big\}
		\\
		&=12\alpha\sigma(\sigma\alpha-\gamma\mu)(4\sigma\alpha-\mu\alpha-\mu\gamma){>0}.
	\end{aligned}
\end{equation*}
Therefore,
\begin{equation*}
	\begin{aligned}
		h'(\beta)\ge h'(\frac{\sigma}{\gamma})=&16\sigma^3\alpha^2\gamma^{-1}(\alpha\sigma+\gamma\mu)-6\alpha\gamma^{-1}\mu\sigma^3(\alpha^2+\gamma^2+10\alpha\gamma)+12\alpha\mu\sigma^2(\alpha+\gamma)(\mu+\sigma)\\
		&-2\mu\sigma\big[\mu\sigma(\alpha+\gamma)^2+2\alpha\gamma(\mu-\sigma)^2\big]
		\\
		=&2\sigma\gamma^{-1}(\sigma\alpha-\gamma\mu)\big\{
		8\sigma\alpha(\sigma\alpha-\gamma\mu)-\mu\alpha(3\sigma\alpha-2\gamma\mu)+\mu\sigma\gamma^2
		\big\}
		\\
		\ge& 2\sigma\gamma^{-1}(\sigma\alpha-\gamma\mu)\big\{
		{4(\mu\alpha+\mu\gamma)}(\sigma\alpha-\gamma\mu)-\mu\alpha(3\sigma\alpha-2\gamma\mu)+\mu\sigma\gamma^2
		\big\}
		\\
		>& 2\sigma\gamma^{-1}(\sigma\alpha-\gamma\mu)\big\{
		{4\mu\alpha}(\sigma\alpha-\gamma\mu)-\mu\alpha(3\sigma\alpha-2\gamma\mu)+\mu\sigma\gamma^2
		\big\}
		\\
		=& 2\sigma\gamma^{-1}(\sigma\alpha-\gamma\mu)\mu \big\{
		\sigma(\alpha^{2}+\gamma^{2})-2\alpha\gamma\mu
		\big\}{>0\quad \mbox{as}\ \sigma\ge\mu.}
	\end{aligned}
\end{equation*}
\end{proof}

\subsection{Estimates of $m(\beta_j)$}
\begin{lemma}\label{est-m-beta}
For the $\beta_{j}$ such that $-\delta_{1}(\beta_{j})=\lambda_{j}$, 
it holds that 
\begin{equation*}
(\beta_j^2\alpha\gamma+\mu\sigma-2\beta_j\mu\gamma)m(\beta_{j})-(\beta_j^2\alpha\gamma+\mu\sigma-2\beta_j\alpha\sigma)<0.
\end{equation*}
\end{lemma}
\begin{proof}
 On the one hand, $m(\beta)<0$ and 
 \begin{equation*}
 	\beta^2\alpha\gamma+\mu\sigma-2\beta\mu\gamma\ge \sigma^2\alpha\gamma^{-1}+\mu\sigma-2\sigma\mu=\sigma\gamma^{-1}(\sigma\alpha-\mu\gamma)>0.
 \end{equation*}
 It is easy to see that $\beta^{2}\alpha\gamma-\mu\sigma>\beta\gamma(\beta\gamma-\sigma)>0$ for all $\beta>\frac{\sigma}{\gamma}$. Besides, 
\begin{equation*}
\begin{aligned}
	m(\beta_{j})
	=-\frac{\lambda_{j}+\sigma b_{\beta_{j}}}{\beta_{j}\gamma b_{\beta_{j}}}
	=-\frac{\lambda_{j}(\beta_{j}^{2}\alpha\gamma-\mu\sigma)}{\mu\beta_{j}\gamma(\beta_{j}\gamma-\sigma)}-\frac{\sigma}{\beta_{j}\gamma}<-\frac{\lambda_{j}}{\mu}-\frac{\sigma}{\beta_{j}\gamma}.
	\end{aligned}
\end{equation*}
Moreover
	\begin{equation*}
	\begin{aligned}
	&
	(\beta^2\alpha\gamma+\mu\sigma-2\beta\mu\gamma)\big(-\frac{\lambda_{j}}{\mu}-\frac{\sigma}{\beta\gamma}\big)-(\beta^2\alpha\gamma+\mu\sigma-2\beta\alpha\sigma)
	=\frac{1}{\mu\beta\gamma}H(\beta),
	\end{aligned}
	\end{equation*}
	where
	\begin{equation*}
		H(\beta)=-(\lambda_{j}+\mu)\beta^3\alpha\gamma^2+\beta^2\gamma\mu(2\gamma\lambda_{j}+\sigma\alpha)+\beta\gamma\mu\sigma(\mu-\lambda_{j})-\mu^{2}\sigma^2.
	\end{equation*}
{	It suffices to prove $H(\beta)<0$.} Indeed,
\begin{equation*}
	\begin{aligned}
		H'(\beta)&=-3(\lambda_{j}+\mu)\beta^2\alpha\gamma^2+2\beta\gamma\mu(2\gamma\lambda_{j}+\sigma\alpha)+\gamma\mu\sigma(\mu-\lambda_{j}).
		\\
		H''(\beta)&=-6(\lambda_{j}+\mu)\beta\alpha\gamma^2+2\gamma\mu(2\gamma\lambda_{j}+\sigma\alpha).
		\\
		H'''(\beta)&=-6(\lambda_{j}+\mu)\alpha\gamma^2<0.
	\end{aligned}
\end{equation*}
Therefore
\begin{equation*}
	\begin{aligned}
		H''(\beta)\le H''(\frac{\sigma}{\gamma})&=-6(\lambda_{j}+\mu)\sigma\alpha\gamma+2\gamma\mu(2\gamma\lambda_{j}+\sigma\alpha)
		\\
		&=2\gamma\big[\lambda_{j}(-3\sigma\alpha+2\mu\gamma)-2\mu\sigma\alpha\big]<0.
	\end{aligned}
\end{equation*}
Consequently,
\begin{equation*}
	\begin{aligned}
		H'(\beta)\le H'(\frac{\sigma}{\gamma})&=-3(\lambda_{j}+\mu)\sigma^2\alpha+2\sigma\mu(2\gamma\lambda_{j}+\sigma\alpha)+\gamma\mu\sigma(\mu-\lambda_{j})
		\\
		&=\sigma(-3\lambda_j-\mu)(\sigma\alpha-\gamma\mu)<0.
	\end{aligned}
\end{equation*}
Accordingly,
\begin{equation*}
	\begin{aligned}
		H(\beta)\le H(\frac{\sigma}{\gamma})&=-(\lambda_{j}+\mu)\sigma^3\alpha\gamma^{-1}+\sigma^2\gamma^{-1}\mu(2\gamma\lambda_{j}+\sigma\alpha)+\mu\sigma^2(\mu-\lambda_{j})-\mu^{2}\sigma^2
		\\
		&=\sigma^2\gamma^{-1}\lambda_{j}\big[-\sigma\alpha+\mu\gamma\big]<0.
	\end{aligned}
\end{equation*}
\end{proof}
\begin{lemma}\label{derivative}
For $\beta_{j}$ sufficiently large, it holds that
 \begin{equation}\label{star}
 	m(\beta_{j})[\mu m(\beta_{j})+\beta
	_{j}\alpha]\cdot \frac{\gamma b_{\beta_{j}}}{\alpha a_{\beta_{j}}}m(\beta_{j})+\sigma+\beta_{j}\gamma m(\beta_{j})
 \end{equation}
 is always positive.
 Furthermore, if we make slight adjustments to the parameters $\alpha$, $\sigma$, and $\gamma$, then \eqref{star} will not equal to zero for any $j$.
\end{lemma}
\begin{proof}
 On the one hand, 
 \begin{equation*}
 	\begin{aligned}
 		 \mu m(\beta_{j})+\beta
 		_{j}\alpha&=\frac{-\mu\lambda_{j}-\mu\sigma b_{\beta_{j}}}{\beta_{j}\gamma b_{\beta_{j}}}+\beta
 		_{j}\alpha=\frac{-\mu\lambda_{j}+(\beta_j^2\alpha\gamma-\mu\sigma) b_{\beta_{j}}}{\beta_{j}\gamma b_{\beta_{j}}}=\frac{-\mu\lambda_{j}+\mu(\beta_j\gamma-\sigma) }{\beta_{j}\gamma b_{\beta_{j}}}.
 		\\
 		\sigma+\beta_{j}\gamma m(\beta_{j})&=\sigma-\beta_{j}\gamma\frac{\lambda_{j}+\sigma b_{\beta_{j}}}{\beta_{j}\gamma b_{\beta_{j}}}=\sigma-\frac{\lambda_{j}+\sigma b_{\beta_{j}}}{ b_{\beta_{j}}}=-\frac{\lambda_{j}}{ b_{\beta_{j}}}.
 	\end{aligned}
 \end{equation*}
 Therefore,
 \begin{equation*}
 	\begin{aligned}
 	&	m(\beta_{j})[\mu m(\beta_{j})+\beta
 	_{j}\alpha]\cdot \frac{\gamma b_{\beta_{j}}}{\alpha a_{\beta_{j}}}m(\beta_{j})+\sigma+\beta_{j}\gamma m(\beta_{j})
 		\\
 		=&m(\beta_{j})\mu\frac{-\lambda_{j}+(\beta_j\gamma-\sigma) }{\beta_{j}\gamma b_{\beta_{j}}}\cdot \frac{\gamma b_{\beta_{j}}}{\alpha a_{\beta_{j}}}m(\beta_{j})-\frac{\lambda_{j}}{ b_{\beta_{j}}}
 		\\
 		=&m^2(\beta_{j})\mu\frac{-\lambda_{j}+(\beta_j\gamma-\sigma) }{\beta_{j}\alpha a_{\beta_{j}}}-\frac{\lambda_{j}}{ b_{\beta_{j}}}
 		\\
 		=&-\frac{1}{\beta_j^2\alpha\gamma a_{\beta_{j}} b_{\beta_{j}}}\big[m(\beta_j)\mu(\lambda_{j}+\sigma b_{\beta_{j}})(\beta_j\gamma-\sigma-\lambda_{j})+\lambda_{j}\beta_j^2\alpha\gamma a_{\beta_{j}}\big].
 	\end{aligned}
 \end{equation*}
 As a matter of fact, { if $\beta_j\gamma-\sigma-\lambda_{j}<0$, then \eqref{star} holds. In the following, we assume $\beta_j\gamma-\sigma-\lambda_{j}\ge0$. }
 Furthermore,
 \begin{equation*}
 	\begin{aligned}
 	&	m(\beta_j)\mu(\lambda_{j}+\sigma b_{\beta_{j}})(\beta_j\gamma-\sigma-\lambda_{j})+\lambda_{j}\beta_j^2\alpha\gamma a_{\beta_{j}}
		=& 
 				\frac{\mu}{\beta_j\gamma b_{\beta_j}(\beta_j^2\alpha\gamma-\mu\sigma)^2}
 			z(\beta_{j}),
 	\end{aligned}
 \end{equation*}
 where
 \begin{equation*}
 \begin{aligned}
 z(\beta_{j})=&
 -\big(\lambda_{j}\beta_j^2\alpha\gamma+\sigma\mu (\beta_j\gamma-\sigma-\lambda_{j})\big)^2(\beta_j\gamma-\sigma-\lambda_{j})+\lambda_{j}\beta_j^3\alpha\gamma^2 \sigma(\beta_j\alpha-\mu) (\beta_j\gamma-\sigma)
 \\ 
 =& \beta_j^5\lambda_{j}\alpha^2\gamma^3 (\sigma-\lambda_{j})+\beta_j^4\lambda_{j}\alpha\gamma^2\big[\lambda_j\alpha(\sigma+\lambda_{j})-\sigma(\sigma\alpha+3\gamma\mu)\big]
+\beta_j^3\gamma^2 \sigma\mu\big\{-\gamma\sigma\mu+\lambda_{j}\alpha(5\sigma+4\lambda_{j})\big\}
 \\
 &+\beta_{j}^{2}\gamma\sigma\mu(\sigma+\lambda_{j})[3\gamma\sigma\mu-2\lambda_{j}\alpha(\sigma+\lambda_{j})]
 -3\beta_j\gamma\sigma^2\mu^2(\sigma+\lambda_j)^2+(\sigma+\lambda_j)^{3}\sigma^{2}\mu^{2}.
 \end{aligned}
 \end{equation*}
{From the above expression,  since $\lambda_j<\sqrt{\mu\sigma}\le \sigma$, we know that for $\beta_j$ large, \eqref{star} is true.}  Furthermore, if we make slight adjustments to the parameters $\alpha$, $\sigma$, $\mu$ and $\gamma$, then $z(\beta_{j})$ will not equal to zero for any $j$.
 \end{proof}

 \subsection{Monotonicity of $\delta_{2}(\beta,\lambda)$ in global bifurcation analysis}
 \begin{prop}\label{mono-lambda}
 For the $\delta_{1}(\beta,\lambda)$ and $\delta_{2}(\beta,\lambda)$ defined in \eqref{eigen-lambda}, it holds that 
 \begin{equation*}
 \delta_{1}(\beta,\lambda)<0,\quad  \frac{\partial}{\partial\lambda}\delta_{2}(\beta,\lambda)<0, \quad   \frac{\partial}{\partial\beta}\delta_{2}(\beta,\lambda)>0.
 \end{equation*}
 \end{prop}
 \begin{proof}
 By straightforward calculations, 
  \begin{equation*}
 	\begin{aligned}
 	\delta_{1}(\beta,\lambda)&=\frac{1}{2\lambda}\Big[	-\mu\lambda	+\beta\alpha b_{\beta}-\sigma\lambda+\beta\gamma a_{\beta}-\sqrt{(\mu\lambda	-\beta\alpha b_{\beta}+\sigma\lambda-\beta\gamma a_{\beta})^2-4(\mu\sigma\lambda^2-\beta\alpha b_{\beta}\sigma\lambda-\beta\gamma a_{\beta}\mu\lambda}\Big]
	\\
	&=\frac{1}{2}\Big[	-\big(\mu	+\sigma-\frac{\beta}{\lambda}(\alpha b_{\beta}+\gamma a_{\beta})\big)-\sqrt{\big(\mu	+\sigma-\frac{\beta}{\lambda}(\alpha b_{\beta}+\gamma a_{\beta})\big)^2-4(\mu\sigma-\frac{1}{\lambda}\beta\alpha b_{\beta}\sigma-\frac{\beta}{\lambda}\gamma a_{\beta}\mu}\Big].
 	\end{aligned}
 \end{equation*}
 In fact, since $\lambda>1$, one has
 \begin{equation*}
 \begin{aligned}
 \mu	+\sigma-\frac{\beta}{\lambda}(\alpha b_{\beta}+\gamma a_{\beta})
=&\frac{1}{\lambda(\beta^2\alpha\gamma-\mu\sigma)}\Big[\lambda(\mu+\sigma)(\beta^2\alpha\gamma-\mu\sigma)-\beta\sigma\gamma(\beta\alpha-\mu)-\beta\mu\alpha(\beta\gamma-\sigma)\Big]
\\
>&\frac{1}{\lambda(\beta^2\alpha\gamma-\mu\sigma)}\Big[(\mu+\sigma)(\beta^2\alpha\gamma-\mu\sigma)-\beta\sigma\gamma(\beta\alpha-\mu)-\beta\mu\alpha(\beta\gamma-\sigma)\Big]
\\
=&\frac{\sigma\mu}{\lambda(\beta^2\alpha\gamma-\mu\sigma)}\big[(\beta\gamma-\sigma)+(\beta\alpha-\mu)\big]>0.
 \end{aligned}
 \end{equation*}
  Moreover,
   \begin{equation*}
 	\begin{aligned}
	\alpha b_{\beta}-\gamma a_{\beta}
	&
	=\frac{1}{\beta^2\alpha\gamma-\mu\sigma}\big[-\beta\alpha\gamma(\sigma-\mu)-\sigma\mu(\alpha-\gamma)\big]<0.
	 \end{aligned}
 \end{equation*}
 And
 \begin{equation*}
 \begin{aligned}
& \big(\mu	+\sigma-\frac{\beta}{\lambda}(\alpha b_{\beta}+\gamma a_{\beta})\big)^2-4(\mu\sigma-\frac{1}{\lambda}\beta\alpha b_{\beta}\sigma-\frac{1}{\lambda}\beta\gamma a_{\beta}\mu)
\\
=&
 \big(\mu	-\sigma-\frac{\beta}{\lambda}(\alpha b_{\beta}-\gamma a_{\beta})\big)^2+4\frac{\beta^2}{\lambda^2}\alpha\gamma b_{\beta} a_{\beta}>0.
 \end{aligned}
 \end{equation*}
 Thus
  \begin{equation*}
 	\begin{aligned}
 		\delta_{1}(\beta,\lambda)
 		&=\frac{1}{2}\Big[	-\big(\mu	+\sigma-\frac{\beta}{\lambda}(\alpha b_{\beta}+\gamma a_{\beta})\big)-\sqrt{\big(\mu	-\sigma-\frac{\beta}{\lambda}(\alpha b_{\beta}-\gamma a_{\beta})\big)^2+4\frac{\beta^2}{\lambda^2}\alpha\gamma b_{\beta} a_{\beta}}\Big]<0.
 	\end{aligned}
 \end{equation*}
 And 
   \begin{equation*}
 	\begin{aligned}
 		\delta_{2}(\beta,\lambda)
		&=\frac{1}{2}\Big[	-(\mu	+\sigma)+\frac{\beta}{\lambda}(\alpha b_{\beta}+\gamma a_{\beta})+\sqrt{\big(\sigma-\mu-\frac{\beta}{\lambda}(\gamma a_{\beta}-\alpha b_{\beta})\big)^2+4\frac{\beta^2}{\lambda^2}\alpha\gamma b_{\beta} a_{\beta}}\Big]
	\\
	&=\frac{1}{2}\Big[	-(\mu	+\sigma)+\frac{\beta}{\lambda}(\alpha b_{\beta}+\gamma a_{\beta})+\sqrt{f(\beta,\lambda)}\Big],
 	\end{aligned}
 \end{equation*}
 where
 \begin{equation*}
 	f(\beta,\lambda):=\big(\sigma-\mu-\frac{\beta}{\lambda}(\gamma a_{\beta}-\alpha b_{\beta})\big)^2+4\frac{\beta^2}{\lambda^2}\alpha\gamma b_{\beta} a_{\beta}.
 \end{equation*}
 Then 
 \begin{equation*}
 	\begin{aligned}
 		\frac{\partial}{\partial\lambda}\delta_{2}(\beta,\lambda)&=-\frac{\beta}{2\lambda^2}(\alpha b_{\beta}+\gamma a_{\beta})+\frac{1}{4\sqrt{f(\beta,\lambda)}}	\frac{\partial}{\partial\lambda}f(\beta,\lambda)
 		\\
 		&=-\frac{\beta}{2\lambda^2}(\alpha b_{\beta}+\gamma a_{\beta})+\frac{1}{4\sqrt{f(\beta,\lambda)}}	\Big[2\big(\sigma-\mu-\frac{\beta}{\lambda}(\gamma a_{\beta}-\alpha b_{\beta})\big)\frac{\beta}{\lambda^2}(\gamma a_{\beta}-\alpha b_{\beta})-8\frac{\beta^2}{\lambda^3}\alpha\gamma b_{\beta} a_{\beta}\Big]
 		\\
 		&=-\frac{\beta}{2\lambda^2}(\alpha b_{\beta}+\gamma a_{\beta})+\frac{1}{2\sqrt{f(\beta,\lambda)}}\frac{\beta}{\lambda^2}	\Big[\big(\sigma-\mu-\frac{\beta}{\lambda}(\gamma a_{\beta}-\alpha b_{\beta})\big)(\gamma a_{\beta}-\alpha b_{\beta})-4\frac{\beta}{\lambda}\alpha\gamma b_{\beta} a_{\beta}\Big]
	\\
		&=
		\frac{\beta}{2\lambda^2\sqrt{f(\beta,\lambda)}}\Big\{-(\alpha b_{\beta}+\gamma a_{\beta})\sqrt{f(\beta,\lambda)}+	\big(\sigma-\mu-\frac{\beta}{\lambda}(\gamma a_{\beta}-\alpha b_{\beta})\big)(\gamma a_{\beta}-\alpha b_{\beta})-4\frac{\beta}{\lambda}\alpha\gamma b_{\beta} a_{\beta}\Big\}.
 	\end{aligned}
 \end{equation*}
{If $\sigma-\mu-\frac{\beta}{\lambda}(\gamma a_{\beta}-\alpha b_{\beta})\le 0$, then $\frac{\partial}{\partial\lambda}\delta_{2,\lambda}(\beta)<0$. If $\sigma-\mu-\frac{\beta}{\lambda}(\gamma a_{\beta}-\alpha b_{\beta})> 0$, then}
  \begin{equation*}
 	\begin{aligned}
 		\frac{\partial}{\partial\lambda}\delta_{2}(\beta,\lambda)
 &\le \frac{\beta}{2\lambda^2\sqrt{f(\beta,\lambda)}}\Big\{-(\alpha b_{\beta}+\gamma a_{\beta})\big(\sigma-\mu-\frac{\beta}{\lambda}(\gamma a_{\beta}-\alpha b_{\beta})\big)+	\big(\sigma-\mu-\frac{\beta}{\lambda}(\gamma a_{\beta}-\alpha b_{\beta})\big)(\gamma a_{\beta}-\alpha b_{\beta})-4\frac{\beta}{\lambda}\alpha\gamma b_{\beta} a_{\beta}\Big\}
 \\
 &= \frac{\beta}{2\lambda^2\sqrt{f(\beta,\lambda)}}\Big\{-2\alpha b_{\beta}\big(\sigma-\mu-\frac{\beta}{\lambda}(\gamma a_{\beta}-\alpha b_{\beta})\big)-4\frac{\beta}{\lambda}\alpha\gamma b_{\beta} a_{\beta}\Big\}<0.
 	\end{aligned}
 \end{equation*}
 Hence 
 \begin{equation*}
 \frac{\partial}{\partial\lambda}\delta_{2}(\beta,\lambda)<0.
	\end{equation*}
	
	Recall that
	\begin{equation*}
	\begin{aligned}
			\frac{\mathrm{d}a_{\beta}}{\mathrm{d} \beta}&=-\frac{\sigma\alpha(\beta^2\alpha\gamma+\mu\sigma-2\beta\gamma\mu)}{(\beta^2\alpha\gamma-\mu\sigma)^2},
		\\
		\frac{\mathrm{d}b_{\beta}}{\mathrm{d} \beta}&=-\frac{\mu\gamma(\beta^2\alpha\gamma+\mu\sigma-2\beta\alpha\sigma)}{(\beta^2\alpha\gamma-\mu\sigma)^2}.
	\end{aligned}
\end{equation*}
On the other hand, 
  \begin{equation*}
 	\begin{aligned}	
	 \frac{\partial}{\partial\beta}\delta_{2}(\beta,\lambda)
	 =&
	 \frac{1}{2\lambda}(\alpha b_{\beta}+\gamma a_{\beta})+\frac{\beta}{2\lambda}\big(\alpha \frac{\mathrm{d}b_{\beta}}{\mathrm{d} \beta}+\gamma \frac{\mathrm{d}a_{\beta}}{\mathrm{d} \beta}\big)+\frac{1}{4\sqrt{f(\beta,\lambda)}}	\frac{\partial}{\partial\beta}f(\beta,\lambda)
	 \\
	 =&
	 \frac{1}{2\lambda}(\alpha b_{\beta}+\gamma a_{\beta})+\frac{\beta}{2\lambda}\big(\alpha \frac{\mathrm{d}b_{\beta}}{\mathrm{d} \beta}+\gamma \frac{\mathrm{d}a_{\beta}}{\mathrm{d} \beta}\big)
	 \\
	 &+\frac{1}{4\sqrt{f(\beta,\lambda)}}\Big\{	2\big(\sigma-\mu-\frac{\beta}{\lambda}(\gamma a_{\beta}-\alpha b_{\beta})\big)\Big[-\frac{1}{\lambda}(\gamma a_{\beta}-\alpha b_{\beta})-\frac{\beta}{\lambda}\big(\gamma \frac{\mathrm{d}a_{\beta}}{\mathrm{d} \beta}-\alpha \frac{\mathrm{d}b_{\beta}}{\mathrm{d} \beta}\big)\Big]
	 \\
	 &\qquad
	 +8\frac{\beta}{\lambda^{2}}\alpha\gamma b_{\beta} a_{\beta}+4\frac{\beta^{2}}{\lambda^2}\alpha\gamma \frac{\mathrm{d}b_{\beta}}{\mathrm{d} \beta}a_{\beta}+4\frac{\beta^{2}}{\lambda^2}\alpha\gamma b_{\beta} \frac{\mathrm{d}a_{\beta}}{\mathrm{d} \beta}\Big\}
	 \\
	  =&
	 \frac{1}{2\lambda}\Big[(\alpha b_{\beta}+\gamma a_{\beta})+\beta\big(\alpha \frac{\mathrm{d}b_{\beta}}{\mathrm{d} \beta}+\gamma \frac{\mathrm{d}a_{\beta}}{\mathrm{d} \beta}\big)\Big]
	 \\
	 &+\frac{1}{2\lambda\sqrt{f(\beta,\lambda)}}\Big\{	-\big(\sigma-\mu-\frac{\beta}{\lambda}(\gamma a_{\beta}-\alpha b_{\beta})\big)\Big[(\gamma a_{\beta}-\alpha b_{\beta})+\beta\big(\gamma \frac{\mathrm{d}a_{\beta}}{\mathrm{d} \beta}-\alpha \frac{\mathrm{d}b_{\beta}}{\mathrm{d} \beta}\big)\Big]
	 \\
	 &\qquad
	 +4\frac{\beta}{\lambda}\alpha\gamma b_{\beta} a_{\beta}+2\frac{\beta^{2}}{\lambda}\alpha\gamma \frac{\mathrm{d}b_{\beta}}{\mathrm{d} \beta}a_{\beta}+2\frac{\beta^{2}}{\lambda}\alpha\gamma b_{\beta} \frac{\mathrm{d}a_{\beta}}{\mathrm{d} \beta}\Big\}.
	\end{aligned}
 \end{equation*}
 Furthermore,
  \begin{equation*}
 	\begin{aligned}	
	&(\alpha b_{\beta}+\gamma a_{\beta})+\beta\big(\alpha \frac{\mathrm{d}b_{\beta}}{\mathrm{d} \beta}+\gamma \frac{\mathrm{d}a_{\beta}}{\mathrm{d} \beta}\big)
	\\
	=& \frac{\mu\alpha(\beta\gamma-\sigma)+\gamma\sigma(\beta\alpha-\mu)}{\beta^2\alpha\gamma-\mu\sigma}-\frac{\beta}{(\beta^2\alpha\gamma-\mu\sigma)^2}\big[\alpha\gamma\mu(\beta^{2}\alpha\gamma+\mu\sigma-2\beta\sigma\alpha)+\gamma\alpha\sigma(\beta^{2}\alpha\gamma+\mu\sigma-2\beta\gamma\mu)\big]
		\\
		=& \frac{\beta\alpha\gamma(\mu+\sigma)-\mu\sigma(\alpha-\gamma)}{\beta^2\alpha\gamma-\mu\sigma}-\frac{\beta\alpha\gamma}{(\beta^2\alpha\gamma-\mu\sigma)^2}\big[(\mu+\sigma)(\beta^{2}\alpha\gamma+\mu\sigma)-2\beta\sigma\mu(\alpha+\gamma\big]
	\\
	=&
	\frac{1}{(\beta^2\alpha\gamma-\mu\sigma)^2}\Big[\beta\alpha\gamma(\mu+\sigma)(\beta^2\alpha\gamma-\mu\sigma)-(\beta^2\alpha\gamma-\mu\sigma)\mu\sigma(\alpha-\gamma)-\beta\alpha\gamma(\mu+\sigma)\big[\beta^{2}\alpha\gamma+\mu\sigma\big]+2\beta^{2}\alpha\gamma\sigma\mu(\alpha+\gamma)\Big]
	\\
	=&\frac{\mu\sigma}{(\beta^2\alpha\gamma-\mu\sigma)^2}\Big[\beta^{2}\alpha\gamma(\alpha+3\gamma)-\beta\alpha\gamma(\mu+\sigma)+\mu\sigma(\alpha-\gamma)\Big]
	\\
	{\ge}&
	\frac{\mu\sigma}{(\beta^2\alpha\gamma-\mu\sigma)^2}\Big[\sigma^{2}\alpha\gamma^{-1}(\alpha+3\gamma)-\sigma\alpha(\mu+\sigma)+\mu\sigma(\alpha-\gamma)\Big]
		\\
	=&	\frac{\mu\sigma^{2}\gamma^{-1}}{(\beta^2\alpha\gamma-\mu\sigma)^2}\Big[\sigma\alpha(\alpha-\gamma)+\gamma(3\sigma\alpha-\mu\gamma)\Big]>0.
		\end{aligned}
 \end{equation*}
 And
   \begin{equation*}
 	\begin{aligned}	
 &-\big(\sigma-\mu-\frac{\beta}{\lambda}(\gamma a_{\beta}-\alpha b_{\beta})\big)\Big[(\gamma a_{\beta}-\alpha b_{\beta})+\beta\big(\gamma \frac{\mathrm{d}a_{\beta}}{\mathrm{d} \beta}-\alpha \frac{\mathrm{d}b_{\beta}}{\mathrm{d} \beta}\big)\Big]
	 +4\frac{\beta}{\lambda}\alpha\gamma b_{\beta} a_{\beta}+2\frac{\beta^{2}}{\lambda}\alpha\gamma \frac{\mathrm{d}b_{\beta}}{\mathrm{d} \beta}a_{\beta}+2\frac{\beta^{2}}{\lambda}\alpha\gamma b_{\beta} \frac{\mathrm{d}a_{\beta}}{\mathrm{d} \beta}
	 \\
	 =&-(\sigma-\mu)(\gamma a_{\beta}-\alpha b_{\beta})+\frac{\beta}{\lambda}(\gamma a_{\beta}-\alpha b_{\beta})^{2}-(\sigma-\mu)\beta(\gamma \frac{\mathrm{d}a_{\beta}}{\mathrm{d} \beta}-\alpha \frac{\mathrm{d}b_{\beta}}{\mathrm{d} \beta})+\frac{\beta^{2}}{\lambda}(\gamma a_{\beta}-\alpha b_{\beta})(\gamma \frac{\mathrm{d}a_{\beta}}{\mathrm{d} \beta}-\alpha \frac{\mathrm{d}b_{\beta}}{\mathrm{d} \beta})
	 \\
	 &
	 +4\frac{\beta}{\lambda}\alpha\gamma b_{\beta} a_{\beta}+2\frac{\beta^{2}}{\lambda}\alpha\gamma \frac{\mathrm{d}b_{\beta}}{\mathrm{d} \beta}a_{\beta}+2\frac{\beta^{2}}{\lambda}\alpha\gamma b_{\beta} \frac{\mathrm{d}a_{\beta}}{\mathrm{d} \beta}
	 \\
	  =&-(\sigma-\mu)(\gamma a_{\beta}-\alpha b_{\beta})+\frac{\beta}{\lambda}(\gamma a_{\beta}+\alpha b_{\beta})^{2}-(\sigma-\mu)\beta\big(\gamma \frac{\mathrm{d}a_{\beta}}{\mathrm{d} \beta}-\alpha \frac{\mathrm{d}b_{\beta}}{\mathrm{d} \beta}\big)+\frac{\beta^{2}}{\lambda}(\gamma a_{\beta}+\alpha b_{\beta})(\gamma \frac{\mathrm{d}a_{\beta}}{\mathrm{d} \beta}+\alpha \frac{\mathrm{d}b_{\beta}}{\mathrm{d} \beta})
	  \\
	   =&\frac{\beta}{\lambda}(\gamma a_{\beta}+\alpha b_{\beta})\Big[(\gamma a_{\beta}+\alpha b_{\beta})+\beta(\gamma \frac{\mathrm{d}a_{\beta}}{\mathrm{d} \beta}+\alpha \frac{\mathrm{d}b_{\beta}}{\mathrm{d} \beta})\Big]
	   -(\sigma-\mu)\Big[(\gamma a_{\beta}-\alpha b_{\beta})
	  +\beta\big(\gamma \frac{\mathrm{d}a_{\beta}}{\mathrm{d} \beta}-\alpha \frac{\mathrm{d}b_{\beta}}{\mathrm{d} \beta}\big)\Big].
\end{aligned}
 \end{equation*}	
 In addition,
  \begin{equation*}
 	\begin{aligned}	
	&(\gamma a_{\beta}-\alpha b_{\beta})
	  +\beta\big(\gamma \frac{\mathrm{d}a_{\beta}}{\mathrm{d} \beta}-\alpha \frac{\mathrm{d}b_{\beta}}{\mathrm{d} \beta}\big)
	  \\
	  =&
 \frac{\gamma\sigma(\beta\alpha-\mu)-\mu\alpha(\beta\gamma-\sigma)}{\beta^2\alpha\gamma-\mu\sigma}
 +\frac{\beta}{(\beta^2\alpha\gamma-\mu\sigma)^2}\big[-\gamma\alpha\sigma(\beta^{2}\alpha\gamma+\mu\sigma-2\beta\gamma\mu)
 +\alpha\gamma\mu(\beta^{2}\alpha\gamma+\mu\sigma-2\beta\sigma\alpha)
 \big]
 \\
   =&
 \frac{\beta\alpha\gamma(\sigma-\mu)+\sigma\mu(\alpha-\gamma)}{\beta^2\alpha\gamma-\mu\sigma}
 +\frac{\beta\alpha\gamma}{(\beta^2\alpha\gamma-\mu\sigma)^2}\big[-(\sigma-\mu)(\beta^{2}\alpha\gamma+\mu\sigma)-2\beta\mu\sigma(\alpha-\gamma)
 \big]
 \\
 =&\frac{1}{(\beta^2\alpha\gamma-\mu\sigma)^{2}}\Big[
 \beta\alpha\gamma(\sigma-\mu)(\beta^2\alpha\gamma-\mu\sigma)+(\beta^2\alpha\gamma-\mu\sigma)\sigma\mu(\alpha-\gamma)-\beta\alpha\gamma(\sigma-\mu)(\beta^{2}\alpha\gamma+\mu\sigma)-2\beta^{2}\alpha\gamma\mu\sigma(\alpha-\gamma)
 \Big]
 \\
  =&
 \frac{1}{(\beta^2\alpha\gamma-\mu\sigma)^{2}}\Big[
-2 \beta\alpha\gamma\mu\sigma(\sigma-\mu)-\sigma^{2}\mu^{2}(\alpha-\gamma)-\beta^{2}\alpha\gamma\mu\sigma(\alpha-\gamma)
 \Big]<0.
	\end{aligned}
 \end{equation*}
 Therefore
 \begin{equation*}
  \frac{\partial}{\partial\beta}\delta_{2}(\beta,\lambda)>0.
 \end{equation*}
 \end{proof}

	\end{document}